\documentclass[11pt,twoside,a4paper]{article}
\usepackage[english]{babel}
\usepackage{a4wide}
\usepackage{graphicx}
\usepackage{amsmath,amssymb,amsthm, amsgen}
\usepackage[ansinew]{inputenc}
\usepackage{listings}
\usepackage{color}
\usepackage[usenames,dvipsnames]{xcolor}
\usepackage{rotating}
\usepackage{hhline}
\usepackage{multirow}
\usepackage{enumerate}
\usepackage{enumitem}
\usepackage[bookmarks]{}
\usepackage{amsfonts}   
\usepackage{dsfont}
\usepackage{tikz}
\usepackage{booktabs}
\usepackage{longtable}
\usepackage{stmaryrd} 
\allowdisplaybreaks  

\newtheorem{thm}{Theorem}[section]

\newtheorem{ass}[thm]{Assumption}
\newtheorem{lem}[thm]{Lemma}

\newcommand{\id}{\operatorname{id}}
\newcommand{\loc}{\operatorname{loc}}
\newcommand{\Lip}{\operatorname{Lip}}

\newcommand{\qtext}[1]{\quad \text{#1} \quad}
\newcommand{\qand}{\qtext{and}}

\newcommand{\supp}{\operatorname{supp}}

\newcommand{\Vwall}{V_{\operatorname{wall}}}
\newcommand{\weakto}{ \rightharpoonup }
\newcommand{\xto}[1]{\xrightarrow{#1}}
\newcommand{\xweakto}[1]{ \stackrel{ #1 }{\rightharpoonup} }

\newcommand{\C}{\mathbb C}

\newcommand{\e}{\varepsilon}
\newcommand{\N}{\mathbb N}

\newcommand{\R}{\mathbb R}

\newcommand{\bx}{\mathbf x}

\newcommand{\by}{\mathbf y}

\newcommand{\cA}{\mathcal A}

\newcommand{\cF}{\mathcal F}

\newcommand{\cL}{\mathcal L}
\newcommand{\cM}{\mathcal M}

\newcommand{\cO}{\mathcal O}
\newcommand{\cP}{\mathcal P}

\newcommand{\cS}{\mathcal S}

\newcommand{\oomega}{\overline \omega}
\newcommand{\opsi}{\overline \psi}
\newcommand{\orho}{\overline \rho}

\newcommand{\rhomin}{ \rho_\ast }
\newcommand{\tilderhomin}{\tilde\rho_\ast}

\DeclareFontFamily{U}{mathx}{\hyphenchar\font45}
\DeclareFontShape{U}{mathx}{m}{n}{
      <5> <6> <7> <8> <9> <10>
      <10.95> <12> <14.4> <17.28> <20.74> <24.88>
      mathx10
      }{}
\DeclareSymbolFont{mathx}{U}{mathx}{m}{n}
\DeclareFontSubstitution{U}{mathx}{m}{n}
\DeclareMathAccent{\widecheck}{0}{mathx}{"71}

\begin{document}

\title{Boundary-layer analysis of repelling particles pushed to an impenetrable barrier}

\author{Patrick van Meurs}

\maketitle

\begin{abstract}
This paper considers the equilibrium positions of $n$ particles in one dimension. Two forces act on the particles; a nonlocal repulsive particle-interaction force and an external force which pushes them to an impenetrable barrier. While the continuum limit as $n \to \infty$ is known for a certain class of potentials, numerical simulations show that a discrete boundary layer appears at the impenetrable barrier, i.e.\ the positions of $o(n)$ particles do not fit to the particle density predicted by the continuum limit. In this paper we establish a first-order $\Gamma$-convergence result which guarantees that these $o(n)$ particles converge to a specific continuum boundary-layer profile. 
\end{abstract}

keywords: {
discrete-to-continuum; boundary layers; 
$\Gamma$-convergence;
$\Gamma$-development.
}

\textbf{MSC}: {
74Q05, 
74G10, 
49J45, 
82C22. 
}


\section{Introduction}

This paper contributes to a recent trend in interacting particle systems which aims to find more detailed information on the particle positions at equilibrium than the information which the continuum limit provides. There are roughly two directions which are currently pursued; convergence rates and particle patterns on mesoscopic scales. The studies on convergence rates (see, e.g.\ 
\cite{BraunOrtner20,
EhrlacherOrtnerShapeev16,
HudsonVanMeursPeletier20ArXiv,
KimuraVanMeurs21,
PronzatoZhigljavsky20,
TanakaSugihara19,
VanMeurs18Proc}) 
aim to find a topology in which the distance between the configuration of $n$ particles and the continuum particle density can be measured and bounded by a small value which vanishes as $n \to \infty$. The studies on particle patterns zoom in on a mesoscopic scale, and reveal how the particles are distributed on this scale, either in the bulk 
(see the paper series started in \cite{PetracheSerfaty14, SandierSerfaty15})
or near the end of the support
\cite{GarroniVanMeursPeletierScardia16,
HallChapmanOckendon10,
HallHudsonVanMeurs18,
Hudson13}.
This paper contributes to the latter, in which case we call the particle pattern a \emph{boundary layer}. More precisely, this paper fills the important gap that was left open in  \cite{GarroniVanMeursPeletierScardia16} on the characterization of boundary layers.

\paragraph{The gap in \cite{GarroniVanMeursPeletierScardia16}} To describe the gap in \cite{GarroniVanMeursPeletierScardia16}, we first recall the corresponding setting. Consider $n+1$ many particles ($n \geq 1$) confined to the half-line
\[
  \Omega = [0, \infty).
\]
We label their positions as $\bx := (x_0, x_1, \ldots, x_n)$ and assume that they are ordered, i.e.\ $\bx \in \Omega_n$, where 
\[
  \Omega_n
  := \{ \bx \in \R^{n+1} : 0 = x_0 < x_1 < \ldots < x_n \}.
\]
The discrete (i.e.\ $n < \infty$) particle interaction energy is given by
\begin{equation}
\label{En}
E_n :\Omega_n \to [0,  \infty), \qquad
 E_n (\bx) := \frac{1}{n^2} \sum_{i=1}^n \sum_{j = 0}^{i-1} \gamma_n V ( \gamma_n ( x_i - x_j ) ) + \frac1n \sum_{i = 0}^{n} U(x_i),
\end{equation}
where $V$ is an interaction potential, $U$ is a confining potential and $\gamma_n > 0$ is a parameter. The double sum accounts for each pair of two particles. Figure \ref{fig:VU} illustrates typical examples for $V$ and $U$. The assumptions and properties of $V$ and $U$ are roughly as follows. $U \in C^1(\Omega)$ with $\min_\Omega U = 0$ and $U(x) \to \infty$ as $x \to \infty$. $V \in L^1(\R)$ is normalized to $\int_\R V = 1$, nonnegative, even and singular at $0$ with the bound 
\begin{equation} \label{V:sing:UB:intro}  
  V(x) \leq C \left\{ \begin{array}{ll}
    |x|^{-a}
    &\text{if } a > 0  \\
    -\log|x|
    &\text{if } a = 0
  \end{array} \right.
  \qquad \text{for all } x \in (0, \tfrac12)
\end{equation}
for some $C > 0$, where $a \in [0, 1)$ is a parameter which bounds the strength of the singularity. We also assume that $V|_{(0,\infty)}$ is non-increasing and convex. We state the precise list of assumptions and resulting properties in Section \ref{s:UV}. With respect to \cite{GarroniVanMeursPeletierScardia16}, we put more assumptions on $V$, but allow for a general confining potential $U$ instead of the specific choice $U(x) = x$. This generalization of $U$ does not result in further complications. Instead, it clarifies the dependence of $U$ on the boundary layer.

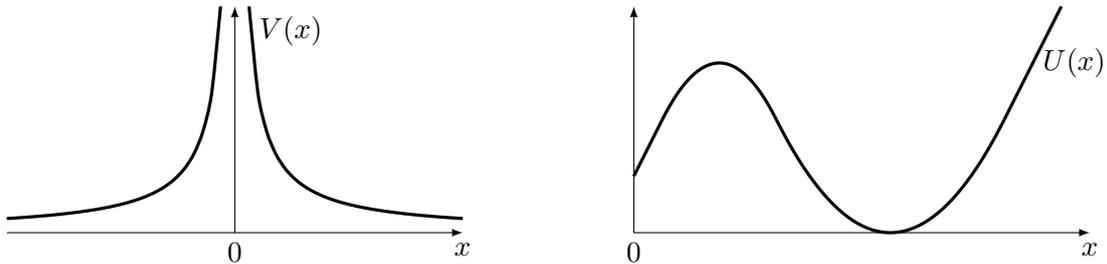
\begin{figure}[h]
\centering
\begin{tikzpicture}[scale=1.5, >= latex]    
\def \w {2}

\begin{scope}[scale = .5]        
\draw[->] (-4,0) --++ (8,0) node[below] {$x$};
\draw[->] (0,0) node[below]{$0$} -- (0,4);


\draw[domain=.25:4, smooth, very thick] plot (\x,{1/\x});
\draw[domain=-4:-.25, smooth, very thick] plot (\x,{-1/\x});

\draw (.25, 4) node[anchor = north west]{$V(x)$};
\end{scope} 

\begin{scope}[shift={(3.5,0)}, xscale=.5, yscale=.25]

\draw[->] (0,0) node[below]{$0$} -- (0,8);
\draw[->] (0,0) -- (8,0) node[below] {$x$};


\draw[very thick] (0,2) -- (.5,4);
\draw[domain=.5:2.5, smooth, very thick] plot (\x,{ 6 - 2*(\x-1.5)^2 });
\draw[domain=2.5:6.5, smooth, very thick] plot (\x,{ (\x-4.5)^2 });
\draw[very thick] (6.5,4) -- (7.5,8);
\draw (7,6) node[right]{$U(x)$};
\end{scope} 
\end{tikzpicture} \\
\caption{Typical examples of $V$ and $U$.}
\label{fig:VU}
\end{figure} 

One particular choice of $V$ which we have in mind is
\begin{equation} \label{Vwall}
  \Vwall (x) = x \coth x - \log| 2 \sinh x |.
\end{equation}
For this potential, $E_n$ is a model for the pile-up of dislocation walls at a lock. We refer to \cite{GarroniVanMeursPeletierScardia16} for the discussion of this model in the literature and its physical relevance. We show in Section \ref{s:UV} that it satisfies all the assumptions that we put on $V$.

The asymptotic behaviour of the parameter $\gamma_n$ in \eqref{En} as $n \to \infty$ plays a decisive role for the limiting energy $E$ of $E_n$ as $n \to \infty$. This can be expected from \eqref{En} by noting that the scaled potential 
\begin{equation} \label{Vgam}
  V_\gamma (x) = \gamma V(\gamma x),
\end{equation}
which has unit integral for all $\gamma > 0$, is squeezed to a delta-peak at $0$ as $\gamma \to \infty$. Hence, as $\gamma$ increases, the particle interactions become more localized. In \cite{GeersPeerlingsPeletierScardia13} the $\Gamma$-limit of $E_n$ is obtained as $n \to \infty$. Depending on the asymptotic behavior of $\gamma_n$, five different limiting energies are obtained: two of these belong to the critical scaling regimes $\gamma_n \to \gamma > 0$ and $\gamma_n / n \to \Gamma > 0$ as $n \to \infty$, and the other three belong to the three regimes separated by these two critical regimes (the outer two regimes require a rescaling of $E_n$ and $\bx$; we refer to \cite{GeersPeerlingsPeletierScardia13} for the details).

In this paper we focus on the regime in between the two critical ones, i.e.
\begin{equation} \label{gamman:L3}
  \lim_{n \to \infty} \gamma_n = \infty
  \qand
  \lim_{n \to \infty} \frac{ \gamma_n }n = 0.
\end{equation}
In this regime the $\Gamma$-limit of $E_n$ (see \cite[Thm.\ 7]{GeersPeerlingsPeletierScardia13}) is given by 
\begin{equation}\label{E}
E : \cP(\Omega) \to [0,\infty], \qquad 
E(\mu) = \frac12 \| \mu \|_{L^2(\Omega)}^2 + \int_\Omega U(x) \, d\mu(x),
\end{equation}
where $\cP(\Omega)$ is the space of probability measures on $\Omega$. The $L^2$-norm is extended to measures (see \eqref{L2:norm:extended} for details) and may be infinite. It is well-known (see, e.g.\ \cite[Thm.\ 2.1]{KinderlehrerStampacchia80} with minor modications to account for $U \notin L^2(\Omega)$) that the minimization problem of $E$ over $\cP(\Omega)$ has a unique minimizer $\mu_* \in L^2(\Omega) \cap \cP(\Omega)$ and that its density $\rhomin$ is characterized by
\begin{equation} \label{EL}
\left\{ \begin{aligned}
  \rhomin + U &\geq C_U \quad \text{on } \Omega \\  
  \rhomin + U &= C_U \quad \text{on } \supp \rhomin, \\ 
\end{aligned} \right. 
\end{equation}
where the constant $C_U > 0$ is such that $\int_\Omega \rhomin(x) \, dx = 1$. Obviously,
\begin{equation} \label{rhoa}
  \rhomin = [C_U - U]^+.
\end{equation}
Figure \ref{fig:U} illustrates $\rhomin$.

\begin{figure}[h]
\centering
\begin{tikzpicture}[scale=1.5, >= latex]    
\def \w {2}

\begin{scope}[xscale=.5, yscale=.25]
\fill[black!10!white] (0,2) --++ (0,2) --++ (.5,0) -- cycle;
\fill[domain=2.5:6.5, smooth, black!10!white] plot (\x,{ (\x-4.5)^2 });

\draw[->] (0,0) node[below]{$0$} -- (0,8);
\draw[->] (0,0) -- (8,0) node[below] {$x$};

\draw (0,4) node[left]{$C_U$} --++ (.5,0);
\draw (2.5,4) --++ (4,0);
\draw[->] (3.5,1) --++ (0,3) node[midway, right]{$\rho_*$};

\draw[thick] (0,2) -- (.5,4);
\draw[domain=.5:2.5, smooth, thick] plot (\x,{ 6 - 2*(\x-1.5)^2 });
\draw[domain=2.5:6.5, smooth, thick] plot (\x,{ (\x-4.5)^2 });
\draw[thick] (6.5,4) -- (7.5,8);
\draw (7,6) node[right]{$U(x)$};
\end{scope} 
\end{tikzpicture} \\
\caption{Geometrical interpretation of $\rho_*$ in \eqref{rhoa}. $\rho_*$ is the height function of the region (colored in gray) with unit area below $C_U$ and above $U$.}
\label{fig:U}
\end{figure}
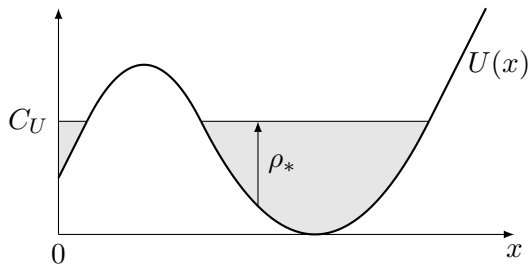 

As pointed out in \cite{GarroniVanMeursPeletierScardia16}, the $\Gamma$-convergence of $E_n$ is not completely satisfactory, because it does not detect any particle patterns on mesoscopic scales. For the scaling regime in \eqref{gamman:L3}, the numerical computations of the minimizer $\bx_*$ in \cite{GarroniVanMeursPeletierScardia16} indicate that $O(n / \gamma_n)$ particles are not distributed according to $\rho_*$; see Figure \ref{fig:GvMPS}. The main result \cite[Thm.\ 1.1]{GarroniVanMeursPeletierScardia16} captures the continuum boundary-layer profile according to which these $O(n / \gamma_n)$ particles are distributed. 
This profile is obtained by firstly proving a first-order $\Gamma$-convergence result for the continuous counterpart of $E_n$ (see \eqref{Egam} below) and by secondly minimizing the first-order $\Gamma$-limit. While Figure \ref{fig:GvMPS} suggests strongly that the obtained boundary-layer profile accurately describes the discrete boundary layer, any proof for this observation was left open. This is the gap in \cite{GarroniVanMeursPeletierScardia16} which we aim to fill in this paper.

\begin{figure}[h]
\centering
\includegraphics[scale=1.2]{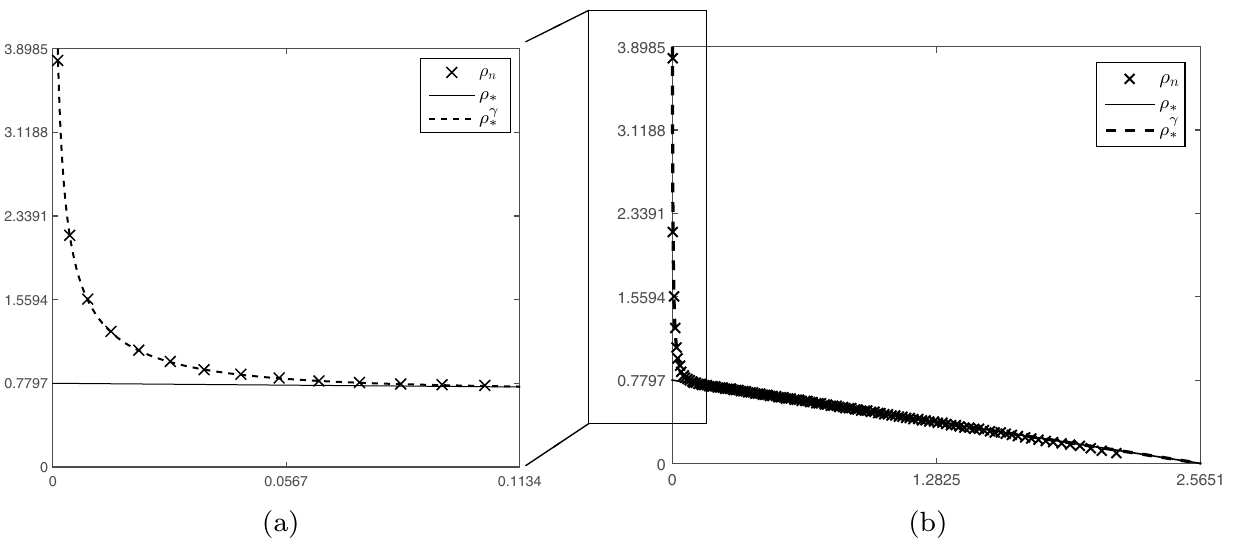}
\caption{This is a copy of \cite[Fig.\ 4]{GarroniVanMeursPeletierScardia16}; copyright by World Scientific Publishing Co., Inc. The Figure illustrates the minimizers $\bx_*$ and $\rho_*$ of respectively $E_n$ and $E$ for the potentials $V$ as in \eqref{Vwall} and $U(x) = x$. (a) is obtained from (b) by zooming in with the scaling operator $(\gamma_n)_\to$ defined in \eqref{scalingOperator}. The crosses $\rho_n$ illustrate the discrete density profile of $\bx_*$. The $x$- and $y$-coordinates of these crosses are respectively $x_{*,i}$ and $\frac2n (x_{*,i+1} - x_{*,i-1})^{-1}$. $\rho_*^{\gamma_n}$ is the continuum boundary-layer profile; see \eqref{nua}. }
\label{fig:GvMPS}
\end{figure} 

\paragraph{The first-order $\Gamma$-convergence result of \cite{GarroniVanMeursPeletierScardia16}} 
In order to describe this paper's first-order $\Gamma$-convergence result which will fill the gap in \cite{GarroniVanMeursPeletierScardia16}, we first recall that of \cite{GarroniVanMeursPeletierScardia16}.
By \cite[Thm.\ 5]{GeersPeerlingsPeletierScardia13} the $\Gamma$-limit of $E_n$ in the regime $\gamma_n \to \gamma > 0$ is given by
\begin{equation}\label{Egam} 
E^\gamma : \cP(\Omega) \to [0,\infty], \qquad
E^\gamma (\mu):= \frac12  \int_\Omega \int_\Omega V_\gamma(x-y) \, d\mu(y)d\mu(x) + \int_\Omega U(x) \, d\mu(x),
\end{equation}
where $V_\gamma$ is defined in \eqref{Vgam}.
The energy $E^\gamma$ is the continuous counterpart of $E_n$ which is considered in \cite{GarroniVanMeursPeletierScardia16}. It $\Gamma$-converges to $E$ as $\gamma \to \infty$ (see \cite[Thm.\ 2.1]{GarroniVanMeursPeletierScardia16}). In particular, this means that there exists a sequence $(\mu^\gamma)_\gamma$ such that
\[
  E^\gamma(\mu^\gamma) - E(\rho_*) = o(1)
  \qquad \text{as } \gamma \to \infty.
\] 
The idea of the authors of \cite{GarroniVanMeursPeletierScardia16} to upgrade this to a first-order $\Gamma$-convergence result was to characterize the $o(1)$-term. They predicted from a priori computations that this term is $O(1/\gamma)$, and that it is easier to replace $E(\rho_*)$ by $E^\gamma(\rho_*)$. This motivated them to consider the functional
\[
  F^\gamma (\mu) := \gamma \big( E^\gamma(\mu) - E^\gamma(\rho_*) \big).
\]
They call $F^\gamma$ the boundary-layer energy. The first-order $\Gamma$-convergence of $E^\gamma$ is simply the (zeroth-order) $\Gamma$-convergence of $F^\gamma$.

For the $\Gamma$-convergence of $F^\gamma$ the topology needed to be chosen carefully. From the formal asymptotics in \cite{Hall11} and their own numerical simulations the authors guessed that the width of the boundary layer is $O(1/\gamma)$. This motivated them to use the following spatial rescaling. For a measure $\mu \in \cP(\Omega)$, let 
\begin{equation}
\label{scalingOperator}
\tilde \mu := \gamma_\to \mu := \gamma \; (\gamma \id)_\#\mu.
\end{equation} 
The inverse scaling is given by
\begin{equation*}
\mu := \gamma_\leftarrow \tilde \mu := \frac1{\gamma} \Big(\frac1{\gamma} \id \Big)_\# \tilde  \mu.
\end{equation*} 
Note that if $\mu$ has a density $\rho$, then the density of $\tilde \mu$ satisfies
\[
  \tilde \rho(x) = \rho(x/\gamma) =: \gamma_\to \rho(x) .
\]
Using this scaling, the authors of \cite{GarroniVanMeursPeletierScardia16} employed the following change of variables:
\begin{equation*} 
\nu^\gamma := \tilde \mu - \tilde \rho_*,
\qquad \mu = \gamma_\leftarrow \nu^\gamma + \rhomin.
\end{equation*}
By subtracting $\rho_*$ the bulk behaviour gets separated from the boundary layer. 

For the signed Radon measures $\nu^\gamma$ with total variation that growths linearly with $\gamma$, the authors used the vague topology. This topology is defined as follows on the space $\cM(\Omega)$ of signed Radon measures on $\Omega$. A sequence $(\nu_\e)_{\e > 0} \subset \cM(\Omega)$ converges to $\nu \in \cM(\Omega)$ vaguely (denoted by $\nu_\e \xweakto v \nu$) as $\e \to 0$ if
\[
  \int_\Omega \varphi \, d\nu_\e
  \xto{\e \to 0} \int_\Omega \varphi \, d\nu
  \qquad \text{for all } \varphi \in C_c(\Omega).
\]

The main result \cite[Thm.\ 1.1]{GarroniVanMeursPeletierScardia16} states that $F^\gamma$ $\Gamma$-converges with respect to the vague topology to a certain limiting boundary-layer energy $F$. This functional $F$ is defined on
\begin{equation} \label{A}
    \mathcal A 
    = \Big\{ \nu \in \cM(\Omega) \mid 
	     \nu^- (dx) \leq \rhomin (0) dx, \ 
	     \sup_{x \geq 0 } \nu^+([x,x+1]) < \infty 
	   \Big\},
\end{equation}
where $\nu^+, \nu^- \geq 0$ are respectively the positive and negative part of $\nu$ such that $\nu = \nu^+ - \nu^-$. While $\nu \in \cA$ may have infinite total variation, we have that $\nu^- \in L^\infty(\Omega)$ and that the local bound on $\nu^+$ is translation invariant. For $\nu \in \cA \cap L^2(\Omega)$,
\begin{equation}
\label{F:formal}
  F(\nu) := 
  \frac12 \int_\Omega \int_\Omega V(x-y) \nu(y) \nu(x) \, d ydx - \rhomin(0) \int_{-\infty}^0 (V * \nu)(x) \, dx.
\end{equation}
It is not obvious to extend this definition to $\nu \in \cA$, because $\nu$ need not be of finite total variation. We recall this extension briefly in Section \ref{s:func:F}. 
Finally, \cite{GarroniVanMeursPeletierScardia16} noted that $F$ has a unique minimizer $\nu_*$ (existence follows from the $\Gamma$-convergence result in \cite{GarroniVanMeursPeletierScardia16} and uniqueness follows from the convexity of $\cA$ and the strict convexity of $F$), and that the continuous boundary-layer profile is given by 
\begin{equation} \label{nua}
  \tilde \rho_*^\gamma := \nu_* + \tilde \rho_*,\qquad 
  \rho_*^\gamma := \gamma_\leftarrow \nu_* + \rho_*.
\end{equation}
Figure \ref{fig:GvMPS} and all other numerical simulations performed in \cite{GarroniVanMeursPeletierScardia16} suggest that $\rho_*^\gamma$ gives a very good prediction for both the bulk and the boundary layer in the minimizer $\bx_*$.
\smallskip

However, the match between $\bx_*$ and $\rho_*^\gamma$ has only been observed and has not been proven. Hence, there is no guarantee that such a match extrapolates to any other choices for the potentials $U$ and $V$ and for the parameter $\gamma_n$. This motivates our aim to establish a first-order $\Gamma$-convergence result for the discrete energy $E_n$ instead of its continuous counterpart $E^{\gamma_n}$.

\paragraph{First-order $\Gamma$-convergence result of $E_n$} 
To establish a first-order $\Gamma$-convergence result for $E_n$, we follow largely the same setup as the one just described. In fact, from Figure \ref{fig:GvMPS} we expect the same limiting boundary-layer energy $F$. Also, there is a close connection between $E_n$ and $E^{\gamma_n}$, which can be seen as follows. Given $\bx \in \Omega_n$, consider the corresponding empirical measure 
\begin{equation}\label{mun}
  \mu_n 
  := \frac1n \sum_{i=0}^n \delta_{x_i}
  \in \frac{n+1}n \cP(\Omega).
\end{equation} 
Then, we can express $E_n(\bx)$ in terms of $\mu_n$ as
\begin{equation}\label{En:mun}
  E_n (\mu_n) := \frac{1}{2} \iint_{\Delta^c} V_{\gamma_n} ( x - y ) \, d\mu_n(y) d\mu_n(x) + \int_{\Omega} U(x) \, d\mu_n(x),
\end{equation} 
where the diagonal
\[
  \Delta = \{ (x,x)^T : x \in \R \} \subset \R^2
\]
is removed from the integration domain to avoid self-interactions. Apart from removing the diagonal, the expressions for $E_n$ and $E^{\gamma_n}$ are the same. Yet, the removal of the diagonal and the difference in the admissible sets on which $E_n$ and $E^{\gamma_n}$ are defined are crucial. Indeed, $(x,y) \mapsto V_{\gamma_n}(x-y)$ concentrates around the diagonal as $n \to \infty$ and thus careful analysis is required.

Following the procedure from \cite{GarroniVanMeursPeletierScardia16}, we consider the blown-up energy difference $\gamma_n [ E_n(\mu_n) - E^{\gamma_n}(\rho_*) ]$ and employ the following change of variables. For $\mu_n$ as in \eqref{mun}, we set
\begin{equation}
\label{nun}
\nu_n := \tilde \mu_n - \tilde \rho_*,
\qquad \mu_n = (\gamma_n)_\leftarrow \nu_n + \rhomin.
\end{equation}
Then, the discrete boundary-layer energy $F_n$ is defined on the admissible set
\begin{equation}
\label{An}
\cA_n
:= 
\bigg\{\nu_n \in \mathcal M(\Omega)
\mid 
\nu_n^- = \tilde \rho_*, \ \exists \, \by \in \Omega_n : \nu_n^+ = \frac{\gamma_n}n \sum_{i=0}^n \delta_{y_i}
 \bigg\}
\end{equation} 
and given by
\begin{equation} \label{Fn}
  F_n(\nu_n) := \gamma_n \big[ E_n \big( (\gamma_n)_\leftarrow \nu_n + \rhomin \big) + E^{\gamma_n}(\rhomin) \big].
\end{equation}
Note that if $\nu_n$ is constructed from $\mu_n$ by \eqref{nun}, then
\begin{equation} \label{Fn:1}
  F_n(\nu_n) = \gamma_n \left(E_n(\mu_n) - E^{\gamma_n}(\rhomin)\right).
\end{equation} 

The main result of this paper in the following $\Gamma$-convergence result of $F_n$:

\begin{thm}\label{t}
Let $U$ and $V$ satisfy Assumptions \ref{a:U} and \ref{a:V}, and let $a \in [0,1)$ be such that \eqref{V:sing:UB:intro} holds. If 
\begin{equation*} 
  1 \ll 
  \gamma_n 
  \ll \left\{ \begin{aligned}
    & n^{\tfrac{1-a}{2-a}}
    && \text{if } 0 < a < 1  \\
    & \sqrt{ \frac{ n}{\log n} }
    && \text{if } a = 0,
  \end{aligned} \right.
\end{equation*}
then any sequence $(\nu_n)_{n=1}^\infty$ with $\nu_n \in \cA_n$ and $\sup_{n \geq 1} F_n(\nu_n) < \infty$ is pre-compact in $\cA$ in the vague topology. Moreover, the functionals $F_n$ $\Gamma$-converge to $F$ with respect to the vague topology, i.e.
\begin{align*}
  &\forall \, \nu \in \cA \
  \forall \, \nu_n \in \cA_n \text{ with } \nu_n \xweakto v \nu : &
  \liminf_{n \to \infty} F_n(\nu_n) &\geq F(\nu) \\
  &\forall \, \nu \in \cA \
  \exists \, \nu_n \in \cA_n \text{ with } \nu_n \xweakto v \nu : &
  \limsup_{n \to \infty} F_n(\nu_n) &\leq F(\nu).
\end{align*}
\end{thm}

More precisely, the assumption on $\gamma_n$ is equivalent to 
\[
  \lim_{n \to \infty} \gamma_n = \infty
  \qand
  \left\{ \begin{aligned}
    \lim_{n \to \infty} \gamma_n n^{-\tfrac{1-a}{2-a}} &= 0
    && \text{if } 0 < a < 1  \\
    \lim_{n \to \infty} \gamma_n \sqrt{ \frac{\log n}n } &= 0 
    && \text{if } a = 0.
  \end{aligned} \right.
\]

The proof of Theorem \ref{t} is given in Section \ref{s:t:pf} with preliminaries in Sections \ref{s:Fn} and \ref{s:func}. It follows the proof in \cite{GarroniVanMeursPeletierScardia16} with major modifications to allow for the discreteness. Here, we briefly describe the main features of Theorem \ref{t} and focus in particular on these major modifications.

First, we recall from \cite{GarroniVanMeursPeletierScardia16} that the expression for $F$ in \eqref{F:formal} arises naturally when the right-hand side in \eqref{Fn} is explicitly expressed in terms of $\nu_n$. In Section \ref{s:Fn} we redo this computation, which in our case deals with the discrete setting and with a general confining potential $U$. 

The main difficulty with respect to \cite{GarroniVanMeursPeletierScardia16} is that the diagonal $\Delta$ is removed from the integration domain (see \eqref{En:mun}) and that the domain of $F_n$ is discrete (i.e.\ the degrees of freedom are empirical measures). To deal with this, we use essentially the particular regularization $V^\beta$ of $V$ constructed in \cite{KimuraVanMeurs21}, which approximates $V$ from below as $\beta \to 0$. Using this regularization, we add and subtract the contribution of the diagonal. By adding the diagonal, we can apply similar arguments as those in \cite{GarroniVanMeursPeletierScardia16} to establish the liminf inequality. However, $\beta$ needs to be chosen carefully. If $\beta$ is too small, then the contribution of the diagonal is too large and may not vanish in the limit. On the other hand, if $\beta$ is too large, then we cannot control the error made by the replacement of $V$ by $V^\beta$. Balancing out these two errors results in the asymptotic upper bound on $\gamma_n$ in Theorem \ref{t}. This bound is a stronger requirement than in \eqref{gamman:L3}, which is sufficient for the (zeroth-order) $\Gamma$-convergence of $E_n$ to $E$.

Establishing the limsup inequality is also significantly more challenging than in \cite{GarroniVanMeursPeletierScardia16}. The discreteness of $\cA_n$ forces us to discretize $\nu$, which was not necessary in the continuous setting in \cite{GarroniVanMeursPeletierScardia16}. Since we blow up the energy difference by the factor $\gamma_n$, we need to show that the discretization error is asymptotically smaller than $1/\gamma_n$. This is much more intricate than for the zeroth-order $\Gamma$-limit of $E_n$ (see \cite[Thm.\ 7]{GeersPeerlingsPeletierScardia13}), where it was sufficient to show that the discretization error simply vanishes as $n \to \infty$. 

\paragraph{Discussion}
In conclusion, Theorem \ref{t} extends its continuous counterpart \cite[Thm.\ 1.1]{GarroniVanMeursPeletierScardia16} (i.e.\ the $\Gamma$-convergence of $F^\gamma$ to $F$) in two manners. First, on a minor note, it allows for a general confining potential $U$. This highlights the fact that the dependence of $F$ on $U$ is restricted to the single value $\rhomin(0) \geq 0$, which depends nonlocally on $U$ (see \eqref{rhoa} and Figure \ref{fig:U}). 
Second, on a major note, Theorem \ref{t} considers the discrete energy $F_n$. As a consequence of Theorem \ref{t}, any sequence of minimizers $\nu_{*,n}$ of $F_n$ converges to $\nu_*$. Since this convergence happens on the mesoscopic scale of the boundary layer, this proves that $\rho_*^{\gamma_n}$ (see \eqref{nua}) indeed describes the boundary layer which appears in $\bx_*$.
This gives the first theoretical motivation for the observations in Figure \ref{fig:GvMPS} and any other numerical computation in \cite{GarroniVanMeursPeletierScardia16} that fits to the regime of $\gamma_n$ assumed in Theorem \ref{t}. This fills the main gap that was left open in \cite{GarroniVanMeursPeletierScardia16}. 

Yet, the story is not complete; \cite{GarroniVanMeursPeletierScardia16} contains a number of conjectures sparked by numerical simulations to which Theorem \ref{t} does not provide an answer. Here, we focus on the main limitation of Theorem \ref{t}, which is the upper bound on $\gamma_n$. Indeed, the numerical simulations in \cite{GarroniVanMeursPeletierScardia16} suggest that $\rho_*^{\gamma_n}$ is the correct boundary-layer profile for the whole regime of $\gamma_n$ in \eqref{gamman:L3}. However, \cite[Table 1]{GarroniVanMeursPeletierScardia16} suggests that the anticipated scaling of the energy difference, i.e.\ $E_n(\bx_*) - E^{\gamma_n}(\rhomin) \sim 1/\gamma_n$, ceases to hold at the upper bound on $\gamma_n$ in Theorem \ref{t}. Hence, this upper bound is not simply an artefact of our proof. Looking deeper into the proof in Section \ref{s:t:pf}, it seems that this upper bound is caused by the contribution to $F_n$ from a narrow region around the diagonal in the double integral in \eqref{En:mun}. A more precise treatment of this diagonal region could perhaps reveal a contribution of the right-hand side in \eqref{Fn:1} which diverges to $\infty$ as $n \to \infty$. Specifying this contribution, subtracting it from $F_n$ and proving $\Gamma$-convergence of the resulting energy functional (provided that his is possible) would reveal that $\rho_*^{\gamma_n}$ remains the correct boundary layer profile beyond the upper bound on $\gamma_n$ in Theorem \ref{t}. Pursuing this direction is beyond our scope.

\paragraph{Organization of the paper}
In Section \ref{s:not} we set the notation. In Section \ref{s:UV} we state the precise assumptions on the potentials $V$ and $U$, and derive further properties that follow from these assumptions. In Section \ref{s:Fn} we rewrite $F_n$ defined in \eqref{Fn} explicitly in terms of $\nu_n$, which will clarify the connection with the expression for $F$ in \eqref{F:formal}. In Section \ref{s:func} we build the functional setting on which our proof of Theorem \ref{t} relies. We also provide several a priori estimates. Finally, Section \ref{s:t:pf} is devoted to the proof of Theorem \ref{t}.

\section{Notation}
\label{s:not}

Here we list some symbols and abbreviations that we use throughout the paper.

\begin{longtable}{lll}
$a$ & smallest constant such that $V(x) \leq C |x|^{-a}$ & As.\ \ref{a:V}\ref{a:V:sing} \\
$\cA_n$, $\cA$ & admissible sets for $F_n$ and $F$ & \eqref{An}, \eqref{A}\\
$\beta$ & regularization parameter for $V$ and $T$ & \eqref{Vbeta} \\
$\gamma_n$ & modelling parameter & Thm.\ \ref{t} \\
$\gamma_\to\mu$, $\gamma_\leftarrow\mu$ & transforms of $\mu$ by scaling space by $\gamma>0$ & \eqref{scalingOperator}\\
$E_n$ & discrete energy & \eqref{En}\\
$E^\gamma$ & $\Gamma$-limit of $E_n$ for $\gamma_n = \gamma$ & \eqref{Egam}\\
$E$ & $\Gamma$-limit of $E_n$ for $1 \ll \gamma_n \ll n$ & \eqref{E}\\
$\widehat f$, $\cF f$ & Fourier transform of $f$; \\
& $(\cF f)(\omega) = \widehat f (\omega) := \int_\R f(x) e^{-2\pi ix\omega} \, dx$ \\
$\cF^{-1} f$ & inverse Fourier transform of $f$;\\
$F_n$ & discrete boundary-layer energy & \eqref{Fn}, \eqref{Fn:full:form} \\
$F$ & continuum boundary-layer energy & \eqref{F:formal}, \eqref{F} \\
$\cL$ & the Lebesgue measure on $\Omega$; $\cL \in \mathcal{M}(\Omega)$ \\
$\mathcal{M}(\Omega)$ & signed Radon measures on $\R$ with support in $\Omega$\\
$\nu^+, \nu^-$ & positive and negative part of a \\ 
& measure $\nu \in \mathcal{M}(\Omega)$; $
\nu^\pm\geq0$ \\
$\mathcal{P}(\Omega)$ & $\mathcal{P}(\Omega) \subset \mathcal{M}(\Omega)$ is the set of probability measures \\
$\rhomin$ & minimizer of $E$ & \eqref{rhoa}\\
$\tilderhomin$ & rescaled version; $\tilderhomin(x) := (\gamma_n)_\to \rho_* (x) = \rho_*(x / \gamma_n)$  \\
$T$ & `convolutional square root' of $V$; $T^2 f = V \ast f$ & \eqref{T}, Lem.\ \ref{l:T} \\
$X_{k,j}$ & Hilbert space; $X_{k,j} \subset L^2(\R)$ & \eqref{Xkj} \\
$\|\cdot\|_p$ &$L^p(\R)$-norm; $1\leq p \leq \infty$.
\end{longtable}

We reserve $c, C > 0$ for generic constants which do not depend on any of the relevant variables. We use $C$ in upper bounds (and think of it as possibly large) and $c$ in lower bounds (and think of it as possibly small). While $c, C$ may vary from line to line, in the same display they refer to the same value. If different constants appear in the same display, we denote them by $C, C', C'', \ldots$.

To avoid clutter, we often omit the integration variable. For instance, we use
\begin{align*}
  \int_\Omega U \, d\rho &:= \int_\Omega U(x) \, d\rho(x) \\
  \int_\R V &:= \int_\R V(x) \, dx \\
  (V * \nu_n)(x) &:= \int_\Omega V(x - y) \, d\nu_n(y),
\end{align*}
and extrapolate this notation to other integrands. Other than the framework of measures, we will also work with distributions. To connect the two notions, we often interpret measures on $\Omega$ as distributions on $\R$ supported in $\Omega$. 

\section{The potentials $U$ and $V$}
\label{s:UV}

To the potential $U$ we add one more assumption to those mentioned in the introduction. We recall that $C_U$ is the constant in \eqref{EL}; see also Figure \ref{fig:U}.

\begin{ass} \label{a:U}
$U \in C^1(\Omega)$ satisfies $\min_\Omega U = 0$ and $U(x) \to \infty$ as $x \to \infty$. Moreover, there exist finitely many disjoint closed intervals $I_1, \ldots, I_m$ such that
\begin{equation}\label{U:supp:assn}
  \supp [C_U - U]^+ = \bigcup_{i=1}^m I_i.
\end{equation}
\end{ass}

The assumption $\min_\Omega U = 0$ is not restrictive, as otherwise one can achieve this by adding a constant to $E_n$. The assumption \eqref{U:supp:assn} is technical; it excludes pathological cases in which the graph of $U$ crosses the value $C_U$ infinitely many times. In fact, for the choice $U(x) = x$ in \cite{GarroniVanMeursPeletierScardia16}, \eqref{U:supp:assn} holds for $m=1$. For $U$ as in Figure \ref{fig:U}, \eqref{U:supp:assn} holds for $m=2$.

In view of \eqref{rhoa}, Assumption \ref{a:U} directly translates to assumptions on $\rhomin$, independent of the assumptions on $V$. Indeed, from \eqref{rhoa} it is clear that Assumption \ref{a:U} implies that $\rhomin$ is Lipschitz continuous, and that $\rho_*'$ is uniformly continuous on $\Omega \setminus \partial (\supp \rho_*)$. Hence, \eqref{U:supp:assn} implies that only at finitely many points $\rho_*$ is not of class $C^1$.
\smallskip

Next we turn to the potential $V$:

\begin{ass} \label{a:V}
$V \in C(\R \setminus \{0\})$ satisfies
\begin{enumerate}[label=(\roman*)]
\item (Evenness). $V:\R\to\R$ is even;
\item \label{a:V:sing} (Singularity). $V(x) \to \infty$ as $x \to 0$, and there exist $C > 0$ and $a \in [0,1)$ such that for all $x \in (0, \frac12]$
\[
  V(x) \leq C \left\{ \begin{array}{ll}
    |x|^{-a}
    &\text{if } a > 0  \\
    -\log|x|
    &\text{if } a = 0;
  \end{array} \right.
\]
\item \label{a:V:cv} (Convexity). $V$ is convex on $(0,\infty)$ and $\lambda$-convex near $x = 0$, i.e. 
\[
  \exists \ \lambda, \delta > 0 :
  x \mapsto V(x) - \frac\lambda2 x^2
  \text{ is convex on } (0, \delta) \text ;
\]
\item \label{a:V:int} (Integrability). $V$ is normalized to $\|V\|_1 = 1$ and has bounded first moment, i.e.\ 
\[
 \int_{\R} |x| V(x) \, dx < \infty;
\]
\item \label{a:V:reg} (Regularity). $V \in W_{\loc}^{2,1}(0,\infty)$ and  $\sqrt{ \widehat{V} } \in W^{2,\infty} (\R)$. 
\end{enumerate}
\end{ass}

First, we mention several properties of $V$ which follow from Assumption \ref{a:V}. The evenness, convexity and integrability imply that $V \geq 0$ is non-increasing on $(0,\infty)$ and that $\widehat V$ is real-valued, nonnegative and even, which is sufficient for Assumption \ref{a:V}\ref{a:V:reg} to be well-defined. A less obvious consequence is Lemma \ref{l:Vhat:LB}.
\begin{lem} \label{l:Vhat:LB}
There exists a constant $c > 0$ such that for all $\omega \in \R$
\[
  v(\omega) := \widehat{V}(\omega) \geq c \min \{ 1, |\omega|^{-2} \}.  
\]
\end{lem}

\begin{proof}
Since $v$ is even and real-valued, it is enough to focus on $\omega > 0$. \cite[(A.3)]{KimuraVanMeurs20DOI} provides the characterization
\[
  v (\omega)
  = \frac1{\pi \omega^3} \sum_{k=0}^\infty \int_0^{\frac12} \bigg( \int_0^{\frac12} V'' \Big( \frac{k + x + y}\omega \Big) dy \bigg) \sin(2\pi x) \, dx.
\]
Since the integrand is nonnegative, we may bound it from below by shrinking the integration domain. Then, on $0 \leq x \leq \frac14$, we bound $\sin(2 \pi x) \geq 4 x$. For $V''$ we note from Assumption \ref{a:V}\ref{a:V:cv} that 
\[
  V''(z) \geq \lambda 1(z < \delta)
  \qquad \text{for all } z > 0,
\]
where $1(P)$ equals $1$ if the statement $P$ is true and $0$ otherwise. Then,
\[
  \widehat V (\omega)
  \geq \frac c{\omega^3} \sum_{k=0}^\infty \int_0^{\frac14} \bigg( \int_0^{\frac14 - x} 1 \big( (k + x + y) < \delta \omega \big) dy \bigg) x \, dx.
\]

We split two cases. If $\omega \leq \frac1{4 \delta}$, then only the term corresponding to $k = 0$ is nonzero, and the right-hand side equals
\begin{align*}
  \frac c{\omega^3} \int_0^{\delta \omega} \int_0^{\delta \omega - x}  x dy  dx
  \geq \delta^3 c'.
\end{align*}
If $\omega > \frac1{4 \delta}$, then we estimate
\begin{align*}
  \widehat V (\omega)
  &\geq \frac c{\omega^3} \sum_{k=0}^\infty \int_0^{\frac14} \bigg( \int_0^{\frac14 - x} 1 \big( (k + \tfrac14) < \delta \omega \big) dy \bigg) x \, dx \\
  &= \frac{c'}{\omega^3} \sum_{k=0}^\infty 1 \big( (k + \tfrac14) < \delta \omega \big)
  = \frac{c'}{\omega^3} \lceil \delta \omega - \tfrac14 \rceil
  \geq \frac45 \delta \frac{c'}{\omega^2}.
\end{align*}
\end{proof}

Next we compare Assumption \ref{a:V} to the assumptions on $V$ made in \cite{GarroniVanMeursPeletierScardia16}, which are weaker. Indeed, in \cite{GarroniVanMeursPeletierScardia16} Assumption \ref{a:V}\ref{a:V:sing} and the regularity on $V$ are not required, and Assumption \ref{a:V}\ref{a:V:cv} is relaxed to the requirement that $V|_{(0,\infty)}$ is non-increasing. While \cite{GarroniVanMeursPeletierScardia16} has a further assumption that $V$ can be approximated from below by a certain class of functions, we show that this holds under Assumption \ref{a:V} by constructing such an approximation explicitly; see \eqref{Vbeta}.

Next we motivate the assumptions which are new with respect to \cite{GarroniVanMeursPeletierScardia16}. We believe that these additional assumptions are minor, and still allow for most of the potentials in practice which satisfy the assumptions in \cite{GarroniVanMeursPeletierScardia16}.
Regarding Assumption \ref{a:V}\ref{a:V:sing}, it is obvious from the bound on $\gamma_n$ in Theorem \ref{t} that a bound on the singularity of $V$ is needed. The requirement $V(x) \to \infty$ as $x \to 0$ might not be necessary. However, including this case in the proof would require a further case splitting. Since we are not aware of any application for this case, we omit it. While the convexity in Assumption \ref{a:V}\ref{a:V:cv} is new, it captures the following three assumptions in \cite{GarroniVanMeursPeletierScardia16}: $\widehat V > 0$ (see the proof of Lemma \ref{l:Vhat:LB}), $V|_{(0,\infty)}$ is non-increasing, and $V$ can be approximated from below by a special class of functions. The local $\lambda$-convexity is a technical addition which simplifies several steps in the proof of Theorem \ref{t}. Finally, in Assumption \ref{a:V}\ref{a:V:reg}, the regularity on $V$ is only a small upgrade of $V \in W_{\loc}^{1,\infty}(0,\infty)$, which follows from convexity. The regularity of $\sqrt v$ is required in \cite{GarroniVanMeursPeletierScardia16} to extend $F$ from $L^2(\Omega)$ to $\cA$. We further exploit this assumption when proving properties of the regularization $V^\beta$. This is the single assumption which can be hard to check in practice.

Finally, we show that $\Vwall$ defined in \eqref{Vwall} satisfies Assumption \ref{a:V}. We recall from \cite{GeersPeerlingsPeletierScardia13} that $\Vwall$ is strictly convex, has a logarithmic singularity (in particular, $a=0$) and exponential tails. Then, the only non-trivial property left to check is the regularity of $\sqrt v$ with $v := \cF \Vwall$. By the strict convexity and the exponential tails of $\Vwall$, it follows that $v \in C^\infty(\R)$ is positive, and thus $\sqrt v \in W_{\loc}^{2, \infty}(\R)$. To extend this to large $\omega$, we recall from \cite[App.\ A.1]{GeersPeerlingsPeletierScardia13} that
\[
  v(\omega) 
  = \frac1{2\omega} \Big( \coth - \frac{\id}{\sinh^2} \Big)(\pi^2 \omega)
  =: \frac{ \varphi(\omega) }\omega.
\]
Note that $\varphi(\omega) \to 1$ and $\varphi'(\omega), \varphi''(\omega) \to 0$ as $\omega \to \infty$. Then, we obtain $\sqrt v \in W^{2, \infty}(\R)$ from
\begin{align*}
  \big( \sqrt v \big)'(\omega)
  &= \frac{ \big( \sqrt \varphi \big)'(\omega) }{\omega^{1/2}} - \frac{ \sqrt \varphi (\omega) }{2 \omega^{3/2}} 
  \xto{ \omega \to \infty } 0 \\
  \big( \sqrt v \big)''(\omega)
  &= \frac{ \big( \sqrt \varphi \big)''(\omega) }{\omega^{1/2}} 
  - \frac{ \big( \sqrt \varphi \big)'(\omega) }{\omega^{3/2}}
  + \frac{ 3 \sqrt \varphi (\omega) }{4 \omega^{5/2}}
  \xto{ \omega \to \infty } 0.
\end{align*}

\section{Explicit expression of $F_n(\nu_n)$} 
\label{s:Fn}

Here we derive an explicit expression for $F_n(\nu_n)$ in terms of $\nu_n$ and motivate the prefactor $\gamma_n$ in \eqref{Fn}. By scaling back, note from \eqref{nun} that $\nu_n \in \cA_n$ can be written as
\[
  \sigma_n 
  := (\gamma_n)_\leftarrow \nu_n
  = \mu_n - \rhomin
\]
for some empirical measure $\mu_n$ of the form \eqref{mun}. 
Then, in view of the right-hand side in \eqref{Fn:1}, we set $V_\gamma(x) := \gamma V(\gamma x)$ and compute
\begin{align*} \notag
  &E_n(\mu_n) - E^{\gamma_n}(\rhomin) \\\notag
  &= \frac12 \iint_{\Delta^c} V_{\gamma_n} ( x - y ) \, d\mu_n(y) d\mu_n(x) 
     - \frac12 \iint_{\Delta^c} V_{\gamma_n} ( x - y ) \, d \rhomin(y) d\rhomin(x) 
     + \int_\Omega U \, d \sigma_n \\ 
  &= \frac{1}{2} \iint_{\Delta^c} V_{\gamma_n} ( x - y ) \, d \sigma_n(y) d \sigma_n(x)
     + \int_\Omega \big[ ( V_{\gamma_n} * \rhomin ) + U \big]\, d \sigma_n,
\end{align*}
where we recall that $\rho_*$ and $\mu_n$ are extended from $\Omega$ to $\R$ by $0$.

Next we rewrite the second term.
With this aim, we set
\begin{equation} \label{orhoa}
  \orho_*(x) := \left\{ \begin{aligned}
    &\rhomin(0)
    && \text{if } x < 0 \\
    &\rhomin(x)
    && \text{if } x \geq 0
  \end{aligned} \right.
\end{equation}
and expand
\begin{multline} \label{pf:Fn:1}
  \int_\Omega \big[ ( V_{\gamma_n} * \rhomin ) + U \big]\, d \sigma_n \\
  = \int_\Omega \big( V_{\gamma_n} * (\rhomin - \orho_*) \big) \, d \sigma_n
    + \int_\Omega \big( (V_{\gamma_n} * \orho_*) - \orho_* \big) \, d \sigma_n
    + \int_\Omega \big( \rhomin + U \big) \, d \sigma_n.
\end{multline}
The first term equals
\begin{equation} \label{force:rewrite:comp}
  - \rhomin(0) \int_\Omega \int_{-\infty}^0 V_{\gamma_n} (x-y) \, dy \, d \sigma_n(x)
  = - \rhomin(0) \int_{-\infty}^0 (V_{\gamma_n} * \sigma_n)(y) \, dy.
\end{equation}
For the integrand of the third term in \eqref{pf:Fn:1}, we note from \eqref{rhoa} that 
\[
  \int_\Omega \big( \rhomin + U \big) \, d \sigma_n
  = \int_\Omega C_U \, d \sigma_n
    + \int_{(\supp \rhomin)^c} (U - C_U) \, d \sigma_n
  =  \frac{C_U}n + \int_{(\supp \rhomin)^c} (U - C_U) \, d \mu_n.
\]

Collecting our computations, we obtain
\begin{multline*}
  E_n(\mu_n) - E^{\gamma_n}(\rhomin)
  = \frac12 \iint_{\Delta^c} V_{\gamma_n} ( x - y ) \, d \sigma_n(y) d \sigma_n(x)
    - \rhomin(0) \int_{-\infty}^0 (V_{\gamma_n} * \sigma_n)(x) \, dx \\
    + \int_\Omega (V_{\gamma_n} - \delta_0) * \orho_* \, d \sigma_n
    + \int_{(\supp \rhomin)^c} (U - C_U) \, d \mu_n
    + \frac{C_U}n.
\end{multline*}
Multiplying by $\gamma_n$ and changing variables (recall $\nu_n = (\gamma_n)_\to \sigma_n$), we get
\begin{multline} \label{Fn:full:form}
  F_n(\nu_n)
  = \frac{1}{2} \iint_{\Delta^c} V ( x - y ) \, d \nu_n(y) d \nu_n(x)
    - \rhomin(0) \int_{-\infty}^0 (V * \nu_n)(x) \, dx \\
    + \int_\Omega \gamma_n (V_{\gamma_n} - \delta_0) * \orho_* \, d \sigma_n
    + \int_{(\supp \tilde \rho_*)^c} \big( U(x/\gamma_n) - C_U \big) \, d \nu_n^+(x)
    +  C_U \frac{\gamma_n}n. 
\end{multline}
For later use we note that the integral in the second term can be rewritten as (recall \eqref{force:rewrite:comp})
\begin{equation} \label{force:rewrite}
  \int_{-\infty}^0 (V * \nu_n)(y) \, dy
  = \int_\Omega \int_x^\infty V (y) \, dy \, d \nu_n(x).
\end{equation}

Note that the first two terms in \eqref{Fn:full:form} resemble the expression of $F$ in \eqref{F}. This motivates the scaling of the energy difference in \eqref{Fn} by $\gamma_n$. We treat the latter three terms in \eqref{Fn:full:form} as error terms when proving Theorem \ref{t}. While the third term obviously vanishes as $n \to \infty$, the other two terms may not for certain sequences $(\nu_n)_n$. We rely on the fact that the second term is nonnegative, and that the integrand in the first term is expected to be small because $V_\gamma \xweakto v \delta_0$ as $\gamma \to \infty$. We give a precise bound later in Lemma \ref{l:Vgam:err:term}. 

\section{Functional setting and preliminaries} 
\label{s:func}

In Section \ref{s:func:F} we recall from \cite{GarroniVanMeursPeletierScardia16} the necessary functional framework to extend the definition of $F$ in \eqref{F:formal} to $\cA$. Since this functional framework also facilitates the statements and proofs of several preliminary estimates, we treat them in the subsequent Section \ref{s:func:Fn}. In this functional framework we identify measures on $\Omega$ as tempered distributions on $\R$ supported in $\Omega$.

\subsection{Proper definition of $F$}
\label{s:func:F}

Since $F$ is the same as in \cite{GarroniVanMeursPeletierScardia16}, we briefly recall the extension of the definition in \eqref{F:formal} on $L^2(\R)$ to $\cA$. Ideally, if there exists a function $u$ such that $V = u*u$, then for $f \in L^2(\R)$ we have for the interaction term that
\begin{equation} \label{pf:37}
  \int_\R (V*f) f
  = \int_\R (u * f)^2 = \| u * f \|_2^2.
\end{equation}
This expression can be extended to distributions by noting that
\begin{equation} \label{L2:norm:extended}  
  \| \xi \|_2^2
  := \sup_{ \varphi \in C_c^\infty(\R) } \big( 2 \langle \xi, \varphi \rangle - \| \varphi \|_2^2 \big) \in [0,\infty].
\end{equation}
However, from Assumption \ref{a:V} it is not clear whether such a function $u$ exists.

One way to avoid characterizing $u$ is to work in Fourier space. Since convolution transforms into multiplication by the Fourier transform, the linear operation of convolving by $u$ turns into multiplication by $\sqrt{ \cF V }$, which is a function due to Assumption \ref{a:V}\ref{a:V:reg}. Precisely, we set
\begin{equation} \label{T}
  T : L^2(\R) \to L^2(\R), \qquad
  T f 
  := \cF^{-1} ( \sqrt{ v } \widehat f ),
\end{equation}
where
$
  v = \cF V
$. Then, by translation to Fourier space, we observe that \eqref{pf:37} turns into
\begin{align} \label{Vff:eq:TfL2}
  \int_\R (V*f) f
  = \| Tf \|_2^2.
\end{align}
Together with the observation in \eqref{force:rewrite} this yields 
\begin{align} \label{F}
  F(\nu) = \frac12 \| T\nu \|_2^2 - \rhomin(0) \int_\Omega g \, d\nu
\end{align}
for all $\nu \in \cA \cap L^2(\R)$, where 
\begin{align} \label{g}
  g(x) := \int_x^\infty V(y)\, dy
  \qquad \text{for all } x \geq 0.
\end{align}

To extend $F$ to $\cA$, we show that the linear term in \eqref{F} is bounded and that the operator $T$ can be extended to $\cA$. We do this in Lemmas \ref{l:T} and \ref{l:fg:nu}. For later use, we state these lemmas in a general form. With this aim, we introduce the Hilbert spaces
\begin{align*} 
  X_{k,j} (\C) &:= \big\{ f \in H^k(\R; \C) : x^{j} f(x) \in L^2(\R; \C) \big\}  \\
  ( f, \phi )_{X_{k,j}} &:= \sum_{\ell = 0}^k \int_\R f^{(\ell)} \overline{\phi^{(\ell)}} + \int_\R x^{2j} f (x) \overline{\phi (x)} \, dx
\end{align*}
for all  $k,j \in \N$.
In case the functions are real-valued, we set
\begin{equation} \label{Xkj}
  X_{k,j} := \Big\{ f \in H^k(\R) : x^{j} f(x) \in L^2(\R) \Big\}.
\end{equation}
Note that $L^2 (\R) = X_{0,0} \supset X_{k,j} \supset \mathcal S (\R)$ for all $k,j \in \N$, where $\mathcal S (\R)$ is the space of Schwarz functions. Note from the Fourier transform property
\[
  \cF( x^j f^{(k)} ) = i^{k + j} (2\pi)^{k-j} \big( \omega^k \widehat f \big)^{(j)}
\]
that $\cF$ is an invertible bounded linear operator from $X_{k,j} (\C)$ to $X_{j,k} (\C)$. 
We further set $X_{k,j}'$ as the dual of $X_{k,j}$ with respect to the $L^2$-topology, and  $\cL(X_{k,j})$ as the space of all bounded linear operators from $X_{k,j}$ to $X_{k,j}$.

\begin{lem}[{\cite[Lem.\ A.2]{GarroniVanMeursPeletierScardia16}}] \label{l:T}
The operator $T$ and the set $\cA$ satisfy
\begin{enumerate}[label=(\roman*)]
  \item \label{l:T:on:A} $\cA \subset X_{1,2}'$; 
  \item \label{l:T:SBLO} For each $k \in \N$ and each $\ell \in \{0,1,2\}$, $T \in \cL (X_{k,j})$ is symmetric, and can be extended to $T \in \cL (X_{k,j}')$ by
  \begin{equation*}
\langle T\xi, f \rangle := \langle \xi, Tf \rangle 
\qquad \text{for all } \xi \in X_{k,j}', \ f \in X_{k,j}.
\end{equation*}
\end{enumerate} 
\end{lem}

\begin{lem}[{\cite[Lem.\ 3.4]{GarroniVanMeursPeletierScardia16}}] \label{l:fg:nu}
There exist constants $C, \e_M > 0$ with 
\[
  \sup_{M \geq 0} \e_M < \infty 
  \qand  
  \e_M \xto{M \to \infty} 0
\]  
  such that for all $\nu \in \cA$ and all $M \geq 0$ it holds that
\begin{align} \label{gnu:UB:l}
  \int_M^\infty g \, d|\nu| 
  &\leq \e_M \sup_{x \geq 0} | \nu | ([x, x+1]),
  \quad \text{where $g$ is as in \eqref{g}, and}  \\ \label{fnu:UB:l}
  \bigg| \int_\Omega f \nu \, \bigg| 
  &\leq C \|f\|_{X_{1,2}} \sup_{x \geq 0} | \nu | ([x, x+1])
  \quad \text{for all } f \in X_{1,2}.
\end{align}
\end{lem}

We remark that while Lemma \ref{l:fg:nu} is a stronger statement than  \cite[Lem.\ 3.4]{GarroniVanMeursPeletierScardia16}, the proof in \cite{GarroniVanMeursPeletierScardia16} instantly implies Lemma \ref{l:fg:nu}.

\subsection{A priori bounds on $\nu_n$}
\label{s:func:Fn}

First we state the counterpart of Lemma \ref{l:T}\ref{l:T:on:A} in the discrete setting. Since Dirac-delta measures are included in $H^{-1}(\R)$, we obtain
\begin{equation} \label{cAn:in:X10p}
  \cA_n \subset H^{-1}(\R) = X_{1,0}'.
\end{equation}

While $\cA_n \subset \cA$, it is pointless to consider $F$ on $\cA_n$, because $F(\nu_n) = \infty$ for each $\nu_n \in \cA_n$ due to the discreteness and the singularity of $V$ at $0$. This is the crucial difference with the continuous setting in \cite{GarroniVanMeursPeletierScardia16} in which $F_n$ is a small perturbation of $F |_{\cA_n}$.  

To deal with the discreteness, we construct a careful regularization of $V$.
With this aim, let $\lambda, \delta > 0$ be as in Assumption \ref{a:V}\ref{a:V:cv}. For $\beta \in (0,\delta)$ we define the regularization
\begin{equation} \label{Vbeta}
  V^\beta(x)
  = \left\{ \begin{aligned}
    & V(\beta) + (x - \beta) V'(\beta) + \frac\lambda2 (x - \beta)^2
    && \text{if } 0 \leq x < \beta \\
    & V(x)
    && \text{if } x \geq \beta,
  \end{aligned} \right.
\end{equation}
and consider the even extension of $V^\beta$ to $\R$.
Figure \ref{fig:Vbeta} illustrates $V^\beta$. Note that on $(0, \beta)$, $V^\beta$ is a parabola which is tangent to $V$ at $x = \beta$. We also set
\[
  W^\beta := V - V^\beta.
\]
Lemma \ref{l:Vbeta} lists several properties of $V^\beta, W^\beta$.

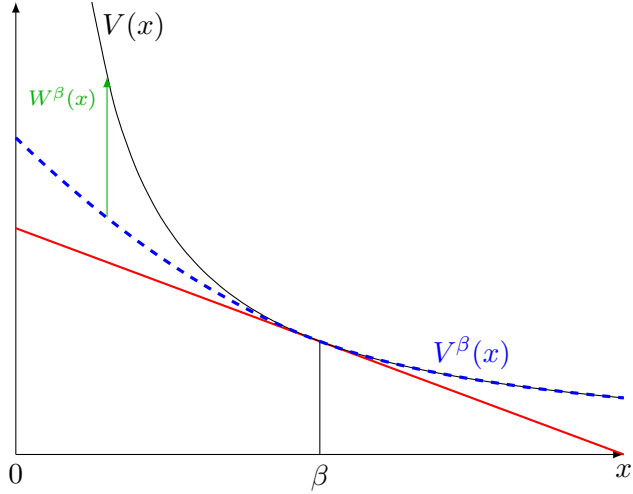
\begin{figure}[h]
\centering
\begin{tikzpicture}[xscale=2, yscale = 3, >= latex]    
\def \w {2}
       
\draw[->] (0,0) --++ (4,0) node[below] {$x$};
\draw[->] (0,0) node[below]{$0$} -- (0,2);

\draw (2,0) node[below]{$\beta$} --++ (0,.5);
\draw[thick, red] (4,0) -- (0,1);
\draw[domain=.5:4, smooth] plot (\x,{1/\x});
\draw[domain=0:2, smooth, very thick, blue, dashed] plot (\x,{ (\x - 2)^2/10 + 1 - \x/4 });
\draw[domain=2:4, smooth, very thick, blue, dashed] plot (\x,{1/\x});
\draw[blue] (3,.333) node[above]{$V^\beta(x)$};
\draw[->, green!70!black] (.6,1.05) -- (.6, 1.67) node[anchor = north east]{\scriptsize $W^\beta(x)$};


\draw (.5, 2) node[anchor = north west]{$V(x)$};
\end{tikzpicture} \\
\caption{Sketch of $V^\beta$. The first two terms in \eqref{Vbeta} for $x < \beta$ describe the tangent line (red) of $V$ at $x = \beta$.}
\label{fig:Vbeta}
\end{figure} 

\begin{lem}[Properties of $V^\beta, W^\beta$] \label{l:Vbeta}
  There exist constants $C,c > 0$ such that for all $\beta$ small enough 
\begin{enumerate}[label=(\roman*)]
  \item \label{l:Vbeta:ptws:bds}(Pointwise bounds). $0 \leq V^\beta \leq V$ and $0 \leq W^\beta \leq V$;
  \item \label{l:Vbeta:cv} (Convexity).  $V^\beta$ and $W^\beta$ are convex and non-increasing on $(0, \infty)$. Moreover, $V^\beta$ satisfies Assumption \ref{a:V}\ref{a:V:cv} with the same constants $\lambda, \delta$;
  \item \label{l:Vbeta:supp:W} (Narrow support). $\supp W^\beta \subset [-\beta, \beta]$;
  \item \label{l:Vbeta:cF} (Fourier transform). $v_\beta := \cF V^\beta$ is real-valued and even, and satisfies 
\begin{equation*} 
  \| v_\beta \|_\infty \leq \| v \|_\infty
  \qand 
  v_\beta(\omega) \geq c \min \{ 1, |\omega|^{-2} \} > 0
  \quad \text{for all } \omega \in \R. 
\end{equation*}
  \item \label{l:Vbeta:L1Linf} ($L^1$, $L^\infty$ bounds). $\| V^\beta \|_\infty = V^\beta(0) \leq  C \beta^{-a}$ and $\| W^\beta \|_1 = \int_{\R} W^\beta \leq C \beta^{1-a}$.
\end{enumerate}
\end{lem}

\begin{proof}
Properties \ref{l:Vbeta:ptws:bds}, \ref{l:Vbeta:cv}, \ref{l:Vbeta:supp:W} and the fact that $v_\beta$ is real-valued and even hold by construction. From these properties, we observe that the proof of Lemma \ref{l:Vhat:LB} also applies to $v_\beta$; this proves the lower bound in \ref{l:Vbeta:cF}. The upper bound follows simply from $\|v_\beta\|_\infty = \|V^\beta\|_1$ and \ref{l:Vbeta:ptws:bds}. The bound on $W^\beta$ in \ref{l:Vbeta:L1Linf} follows from \ref{l:Vbeta:supp:W} through
\[
  \int_{\R} W^\beta 
  = \int_{-\beta}^\beta W^\beta
  \leq \int_{-\beta}^\beta V
  \leq C \int_{-\beta}^\beta |x|^{-a} \, dx
  = C' \beta^{1-a}.
\]
To prove the bound on $V^\beta(0)$ in \ref{l:Vbeta:L1Linf}, we observe from \eqref{Vbeta} that
\[
  V^\beta(0)
  = V(\beta) - \beta V'(\beta) + \frac\lambda2 \beta^2
  \leq C \beta^{-a} + \beta |V'(\beta)|.
\]
By the Mean Value Theorem and the fact that $|V'|$ is non-increasing, we get
\[
  |V'(\beta)|
  \leq \frac{ V(\beta/2) - V(\beta) }{ \beta/2 }
  \leq 2 \frac{ C \beta^{-a} - 0 }{ \beta }
  = 2C \beta^{-1-a}.
\]
This completes the proof of Lemma \ref{l:Vbeta}.
\end{proof}

Lemma \ref{l:Vbeta}\ref{l:Vbeta:cF} allows us to define, similarly to $V$, the convolutional square root operator
\[
  T^\beta : L^2(\R) \to L^2(\R), \qquad
  T^\beta f 
  := \cF^{-1} ( \sqrt{ v_\beta } \widehat f ),
\]
where $v_\beta = \cF V^\beta$. As in \eqref{Vff:eq:TfL2}, it is easy to see (in Fourier space) that $V^\beta * f = T^\beta T^\beta f$ and 
\begin{equation} \label{pf:35}
  \int_\R (V^\beta * f) f = \| T^\beta f \|_2^2.
\end{equation}
Lemma \ref{l:Tbeta} states further properties of $T^\beta$.

\begin{lem} \label{l:Tbeta} The operator $T^\beta$ defined above satisfies
\begin{enumerate}[label=(\roman*)] 
  \item \label{l:Tbeta:BLSO} For each $k \in \N$ and all $\beta$ small enough, $T^\beta \in \cL( X_{k,0} )$ is symmetric, and can be extended to $T^\beta \in \cL(X_{k,0}')$ as in Lemma \ref{l:T};
  \item \label{l:Tbeta:LXk0:bd} $\displaystyle \limsup_{ \beta \to 0} \|T^\beta\|_{\cL( X_{k,0} )} < \infty$ for each $k \in \N$; 
  \item \label{l:Tbeta:Tnu2:Vnunu} For all $\beta$ small enough, each $n \geq 1$ and all $\nu_n \in \cA_n$, there holds $T^\beta \nu_n \in L^2(\R)$,
  \begin{equation*}
    V^\beta \ast \nu_n = T^\beta T^\beta \nu_n
    \qand
    \int_\Omega (V^\beta \ast \nu_n) \, d \nu_n
    = \| T^\beta \nu_n \|_2^2 \text ;
  \end{equation*}
  \item \label{l:Tbeta:TTbetag:bd} There exists $\beta_0 > 0$ such that for all Schwarz functions $f \in \cS (\R)$ there exists $C_f > 0$ such that for all $\beta \in (0,\beta_0)$
  \[
    \| (T - T^\beta) f \|_{X_{1,2}} 
    \leq C_f \beta^{1-a}.
  \] 
\end{enumerate} 
\end{lem}

\begin{proof}
Since $\sqrt{ v_\beta } = (\cF V^\beta)^{1/2}$ is even and bounded, \ref{l:Tbeta:BLSO} and \ref{l:Tbeta:LXk0:bd} follow from the same argument used in the proof of \cite[Lem.\ A.2(ii)]{GarroniVanMeursPeletierScardia16}. In particular, 
\[
  \|T^\beta\|_{\cL( X_{k,0} )}
  \leq C \| \sqrt{ v_\beta } \|_\infty,
\]
for which Lemma \ref{l:Vbeta}\ref{l:Vbeta:cF} provides a sufficient bound.

Next we proof \ref{l:Tbeta:Tnu2:Vnunu}. It is clear that $V^\beta \ast \nu_n = T^\beta T^\beta \nu_n$ holds in distributional sense. Note that \eqref{pf:35} defines a seminorm, which can be extended to distributions similarly to \eqref{L2:norm:extended}. Using this
and applying \ref{l:Tbeta:BLSO}, we write
\begin{align*}
  \infty > \int_\Omega (V^\beta \ast \nu_n) \, d \nu_n
  &= \sup_{\varphi \in X_{1,0}} \bigg( 2\langle V^\beta * \varphi, \nu_n \rangle - \int_\Omega (V^\beta \ast \varphi) \, d \varphi \bigg) \\
  &= \sup_{\varphi \in X_{1,0}} \big( 2\langle T^\beta T^\beta \varphi, \nu_n \rangle - \| T^\beta \varphi \|_2^2 \big) \\
  &= \sup_{\psi \in T^\beta X_{1,0}} \big( 2\langle \psi, T^\beta\nu_n \rangle - \| \psi \|_2^2 \big),
\end{align*}
where $T^\beta X_{1,0} := \{ T^\beta f \mid f \in X_{1,0} \}$.
We claim that $T^\beta X_{1,0}$ is dense in $X_{1,0}$. From this claim, \ref{l:Tbeta:Tnu2:Vnunu} follows by applying \eqref{L2:norm:extended}. To prove the claim, we note that, by translating it to Fourier space, it is equivalent to the claim that $\{ \sqrt{ v_\beta } f \mid f \in X_{0,1} \}$ is dense in $X_{0,1}$. This is easily seen to be true; for $f \in X_{0,1}$, set $f_k := f |_{(-k,k)}$ and note from the lower bound in Lemma \ref{l:Vbeta}\ref{l:Vbeta:cF} that $f_k / \sqrt{ v_\beta } \in X_{0,1}$ for all $k \in \N$.

Finally we prove \ref{l:Tbeta:TTbetag:bd}.
Note that
\begin{equation} \label{pf:33}    
  \| (T - T^\beta) f \|_{X_{1,2}} 
  \leq C \big\| \cF \big( (T - T^\beta) f \big) \big\|_{X_{2,1}}. 
\end{equation}
Recalling $v = \cF V$, we set $w_\beta := \cF W^\beta = v - v_\beta$ and compute
\[
  \cF \big( (T - T^\beta) f \big)
  = (\sqrt v - \sqrt{v_\beta}) \widehat f
  = \frac{ w_\beta }{\sqrt v + \sqrt{v_\beta}} \widehat f
  =: w_\beta u_\beta \widehat f,
\]
where $u_\beta = (\sqrt v + \sqrt{v_\beta})^{-1/2}$. We observe from Lemma \ref{l:Vbeta}\ref{l:Vbeta:ptws:bds},\ref{l:Vbeta:supp:W} that for any $j \in \N$
\[
  \| w_\beta^{(j)} \|_\infty 
  = (2\pi)^j \| \cF( x^j W^\beta ) \|_\infty
  = (2\pi)^j \int_\R |x|^j W^\beta(x) \, dx
  \leq C_j \int_0^\beta x^{j-a} \, dx
  = C_j' \beta^{j+1-a}.
\]
This will eventually result in the prefactor in \ref{l:Tbeta:TTbetag:bd}. 

Next we bound $u_\beta$. We recall from Assumption \ref{a:V}\ref{a:V:reg} that
\[
  \sqrt v, (\sqrt v)' ,(\sqrt v)''  \in L^\infty(\R)
\]
and from Lemmas \ref{l:Vhat:LB} and \ref{l:Vbeta}\ref{l:Vbeta:cF} that
\[
  \min \big\{ \sqrt{ v }, \sqrt{ v_\beta } \big\} \geq c \min \Big\{ 1,  \frac1{|\omega|} \Big\}.
\]
Here and henceforth, we often abuse notation by removing the variable $\omega$ from the notation for functions in Fourier space.
Then,
\[
  u_\beta
  = \frac1{\sqrt v + \sqrt{v_\beta}}
  \leq C(1 + |\omega|)
  =: C \oomega^1,
\]
where $C$ is independent of $\beta$ and 
\[
  \oomega^k := (1 + |\omega|^k)
  \qquad \text{for all } k \in \N.
\]
Using this, we obtain for the $X_{0,1}$-part of the norm in \eqref{pf:33} that
\[
  \big\| \cF \big( (T - T^\beta) f \big) \big\|_{X_{0,1}}
  = \big\| w_\beta u_\beta \widehat f \big\|_{X_{0,1}}
  \leq \big\| \oomega^1 w_\beta u_\beta \widehat f \big\|_2
  \leq C \beta^{1-a} \| \oomega^2 \widehat f \|_2
  = C_f \beta^{1-a}.
\]

For the $H^1$-part of the norm, we compute
\begin{equation} \label{pf:34}
  u_\beta'
  = \frac{-1}{(\sqrt v + \sqrt{v_\beta})^2} \Big( (\sqrt v)' + \frac{v_\beta'}{2 \sqrt{v_\beta}} \Big)
  = - u_\beta^2 \bigg( (\sqrt v)' + \frac{v' - w_\beta'}{2 \sqrt{v_\beta}} \bigg). 
\end{equation} 
Writing $v' = 2 \sqrt v (\sqrt v)' \in L^\infty(\R)$, all terms in the expression for $u_\beta'$ can be bounded from the estimates above. This yields
\[
  |u_\beta'|
  \leq C \oomega^2 \Big( C' + ( C'' + \beta^{2-a} ) \oomega^1 \Big)
  \leq C''' \oomega^3.
\]
Hence,
\begin{align*}
  \big\| \big( \cF \big( (T - T^\beta) f \big) \big)' \big\|_2
  &\leq \big\| w_\beta' u_\beta \widehat f \big\|_2 + \big\| w_\beta u_\beta' \widehat f \big\|_2 + \big\| w_\beta u_\beta (\widehat f)' \big\|_2 \\
  &\leq C \beta^{1-a} \Big( \beta \big\| \oomega^1 \widehat f \big\|_2 + \big\| \oomega^3 \widehat f \big\|_2 + \big\| \oomega^1 (\widehat f)' \big\|_2 \Big) 
  \leq C_f \beta^{1-a}.
\end{align*}

Finally we bound the $H^2$-part of the norm. By the estimates obtained so far,
\begin{align} \notag
  \big\| \big( \cF \big( (T - T^\beta) f \big) \big)'' \big\|_2
  &\leq \big\| w_\beta (u_\beta \widehat f)'' \big\|_2 + 2 \big\| w_\beta' (u_\beta \widehat f)' \big\|_2 + \big\| w_\beta'' u_\beta \widehat f \big\|_2 \\\label{pf:36}
  &\leq C \beta^{1-a} \big\| (u_\beta \widehat f)'' \big\|_2 + C_f \beta^{2-a}.
\end{align}
In preparation for estimating the second derivative, we compute (using the expression in \eqref{pf:34})
\begin{align*}  
  u_\beta''
  = - 2 u_\beta u_\beta' \bigg( (\sqrt v)' + \frac{v' - w_\beta'}{2 \sqrt{v_\beta}} \bigg)
    - u_\beta^2 \bigg( (\sqrt v)'' + \frac{v'' - w_\beta''}{2 \sqrt{v_\beta}} - \frac{(v' - w_\beta')^2}{4 v_\beta^{3/2}} \bigg).
\end{align*}
Then, rewriting
\[
  v'' 
  = ((\sqrt v)^2)''
  = 2 \sqrt v (\sqrt v)'' + 2 ((\sqrt v)')^2
  \in  L^\infty(\R),
\]
we estimate all terms in the expression for $u_\beta''$ by the bounds obtained above. This yields
\begin{align*}
  |u_\beta''|
  \leq C \oomega^1 \oomega^3 (1 + \oomega^1) + C' \oomega^2 (1 + \oomega^1 + \oomega^3)
  \leq C'' \oomega^5. 
\end{align*}
Returning to \eqref{pf:36}, we obtain
\begin{align*}
  \big\| (u_\beta \widehat f)'' \big\|_2
  &\leq \big\| u_\beta'' \widehat f \big\|_2 + 2 \big\| u_\beta' (\widehat f)' \big\|_2 + \big\| u_\beta (\widehat f)'' \big\|_2 \\
  &\leq C \big( \big\| \oomega^5 \widehat f \big\|_2 + \big\| \oomega^3 (\widehat f)' \big\|_2 + \big\| \oomega^1 (\widehat f)'' \big\|_2 \big) 
  \leq C_f.
\end{align*}
Plugging this estimate into \eqref{pf:36} completes the proof of \ref{l:Tbeta:TTbetag:bd}.
\end{proof}

This completes the preliminaries on the regularization $V^\beta$. Next we apply them to construct tools for the proof of Theorem \ref{t}. The first of these tools is a crucial estimate in the proof of compactness. It is the discrete counterpart of \cite[Lem.\ 3.3]{GarroniVanMeursPeletierScardia16}.

\begin{lem} \label{l:nun:bds}
There exists a constant $C > 0$ such that for all $\beta > 0$ small enough and all $n \geq 1$
\begin{align*} 
  \sup_{x \geq 0} \nu_n^+ ([x, x+1]) 
  + \int_{-\infty}^0 V * \nu_n
  \leq 
  C (  \| T^\beta \nu_n \|_{L^2(\R)}  + 1 ).
\end{align*} 
\end{lem}

\begin{proof}
The proof is a modification of the proof of \cite[Lem.\ 3.3]{GarroniVanMeursPeletierScardia16}. Let $x \geq 0$ be arbitrary. Using that $V^\beta \geq 0$ is non-increasing on $[0,\infty)$, we obtain
\begin{multline*}
  \int_x^{x+1} V^\beta * \nu_n^+
  = \int_\Omega \int_x^{x+1} V^\beta (y-z) \, dy d\nu_n^+(z)
  \geq \int_{[x, x+1]} \int_x^{x+1} V^\beta (y-z) \, dy d\nu_n^+(z) \\
  \geq \int_{[x, x+1]} \bigg( \int_0^1 V^\beta (y) \, dy \bigg) d\nu_n^+(z)
  = \bigg( \int_0^1 V^\beta \bigg) \nu_n^+ ([x, x+1]).
\end{multline*}
Taking $\beta$ small enough such that $\int_0^1 V^\beta \geq \frac12 \int_0^1 V > 0$, Lemma \ref{l:Tbeta}\ref{l:Tbeta:Tnu2:Vnunu} implies
\begin{align*}
  \nu_n^+ ([x, x+1])
  \leq \frac2{\int_0^1 V} \int_x^{x+1} V^\beta * \nu_n^+
  = \frac2{\int_0^1 V} \int_x^{x+1} (T^\beta T^\beta \nu_n) 
    + \frac2{\int_0^1 V} \int_x^{x+1} V^\beta * \nu_n^-.
\end{align*}
Then, applying the Cauchy-Schwarz Inequality and Lemma \ref{l:Tbeta}\ref{l:Tbeta:LXk0:bd}
\begin{align} \notag
  \nu_n^+ ([x, x+1])
  &\leq C \| T^\beta T^\beta \nu_n \|_{L^2(\R)} + C \int_x^{x+1} \int_0^{\gamma_n} V^\beta(y-z) \tilde \rho_*(z) \, dz dy \\\label{pf:32}
  &\leq C' \| T^\beta \nu_n \|_{L^2(\R)} + C \bigg( \int_\R V \bigg) \| \rho_* \|_\infty,
\end{align}
which shows the desired estimate for the first term in the display in Lemma \ref{l:nun:bds}.

For the second term, we use that $V \geq 0$ is non-increasing on $(0, \infty)$ together with \eqref{pf:32} to estimate
\begin{align*}
  \int_{-\infty}^0 V * \nu_n
  &\leq \int_{-\infty}^0 \int_\Omega V(y-z) d \nu_n^+(z) dy \\
  &\leq \int_{-\infty}^0 \sum_{k=0}^\infty V(y-k) \, \nu_n^+([k, k+1]) dy \\
  &\leq  C (  \| T^\beta \nu_n \|_{L^2(\R)}  + 1 ) \sum_{k=0}^\infty \int_k^\infty V,
\end{align*}
where the sum is finite due to Assumption \ref{a:V}\ref{a:V:int}.
\end{proof}

The local bound on $\nu_n^+$ in Lemma \ref{l:nun:bds} turns out useful when passing to the limit $n \to \infty$ in the second and third term in the expression for $F_n$ in \eqref{Fn:full:form}. This is made precise in the following two lemmas.

\begin{lem} \label{l:Vgam:err:term}
There exists $C > 0$ such that for all $n \geq 1$ and all $\nu_n \in \cA_n$
\[
  \bigg| \int_\Omega \gamma_n (V_{\gamma_n} - \delta_0) * \orho_* \, d \sigma_n \bigg|
  \leq C,
\]
where $\sigma_n = (\gamma_n)_\leftarrow \nu_n$ and $\orho_*$ is defined in \eqref{orhoa}.
If
\[
  \sup_{n \geq 1} \sup_{x \geq 0} \nu_n^+ ([x, x+1]) < \infty,
\]
then
\begin{align*} 
\int_\Omega \gamma_n (V_{\gamma_n} - \delta_0) * \orho_* \, d \sigma_n
  &\xto{n \to \infty} 0.
\end{align*}
\end{lem}

\begin{proof}
Let $u_n := \gamma_n (V_{\gamma_n} - \delta_0) * \orho_*$. Since $\orho_*$ is Lipschitz continuous, we obtain
\begin{multline} \label{pf:26}
  \| u_n \|_{L^\infty(\R)}
  = \sup_{x \in \R} \bigg| \int_\R \gamma_n \big( \orho_*(x-y) - \orho_*(x) \big) V(\gamma_n y) \gamma_n \, dy \bigg| \\
  \leq L \int_\R |\gamma_n y| V(\gamma_n y) \gamma_n \, dy
  = L \int_\R |z| V(z) \, dz,
\end{multline}
which is bounded due to Assumption \ref{a:V}\ref{a:V:int}. Hence,
\begin{equation} \label{pf:27}  
  \bigg| \int_\Omega u_n \, d \sigma_n \bigg| 
  \leq \| u_n \|_\infty |\sigma_n|(\Omega)
  \leq C (2 + n^{-1}).
\end{equation}

To prove the convergence statement in Lemma \ref{l:Vgam:err:term}, we improve the bound on $u_n$ in \eqref{pf:26}. Let $\e > 0$ be arbitrary. By \eqref{rhoa} and \eqref{U:supp:assn} there exist finitely many points $y_1, \ldots, y_\ell$ such that $\orho_*'$ is uniformly continuous on $\R \setminus \{y_1, \ldots, y_\ell\}$. Let $I_i := (y_i - \e_0, y_i + \e_0)$ where $\e_0 = \frac\e{2 \ell}$. Then, for all $x \in \cO^c := \R \setminus \cO$,
\[
  \orho_*(x-y) - \orho_*(x)
  = -\rho_*'(x) y + r_x(y)
  \qquad \text{for all } y \in \R,
\]  
where the bounded function $y \mapsto r_x(y) / y$ vanishes as $y \to 0$ uniformly in $x \in \cO^c$. Using this, we estimate, similarly to \eqref{pf:26},
\begin{align*} 
  \| u_n \|_{L^\infty(\cO^c)}
  &= \sup_{x \in \cO^c} \bigg| \int_\R \gamma_n \big( \rho_*'(x) y + r_x(y) \big) V(\gamma_n y) \gamma_n \, dy \bigg| \\
  &\leq \gamma_n \|\rho_*'\|_\infty \bigg| \int_\R y V_{\gamma_n}( y) \, dy \bigg|
  + \sup_{x \in \cO^c} \int_{\{ |z| < \sqrt{ \gamma_n } \}} \Big| \frac{r_x(z / \gamma_n)}{z / \gamma_n} \Big| |z| V(z) \, dz \\
  &\qquad + \sup_{x \in \cO^c} \int_{\{ |z| > \sqrt{ \gamma_n } \}} \Big| \frac{r_x(z / \gamma_n)}{z / \gamma_n} \Big| |z| V(z) \, dz \\
  &\leq 0
  + \bigg( \sup_{x \in \cO^c} \sup_{|z| < \sqrt{ \gamma_n } } \Big| \frac{r_x(z / \gamma_n)}{z / \gamma_n} \Big| \bigg) \int_\R |z| V(z) \, dz
  + C \int_{\{ |z| > \sqrt{ \gamma_n } \}} |z| V(z) \, dz \\
  &\xto{n \to \infty} 0.
\end{align*}
Using this, we sharpen the bound in \eqref{pf:27} by 
\[
  \bigg| \int_\Omega u_n \, d \sigma_n \bigg| 
  \leq \| u_n \|_{L^\infty(\R)} |\sigma_n|(\cO) 
       + \| u_n \|_{L^\infty(\cO^c)} |\sigma_n|(\Omega).
\]
The second term vanishes as $n \to \infty$. For the first term, we note that
\[
  |\sigma_n|(\cO)
  \leq \frac1{\gamma_n} \sum_{i=1}^\ell |\nu_n|(\gamma_n I_i).
\]
Covering each interval $\gamma_n I_i$ with $\lceil \e \gamma_n / \ell \rceil$ many intervals of length $1$, we use the given bound on $\nu_n^+$ to continue this estimate by
\[
  \frac1{\gamma_n} \sum_{i=1}^\ell |\nu_n|(\gamma_n I_i)
  \leq \frac1{\gamma_n} \sum_{i=1}^\ell C \Big( \frac{ \e \gamma_n }{\ell} +1 \Big)
  = C' \e.
\]
In conclusion,
\[
  \limsup_{n \to \infty } \bigg| \int_\Omega u_n \, d \sigma_n \bigg| 
  \leq C \e.
\]
Since $\e > 0$ is arbitrary, Lemma \ref{l:Vgam:err:term} follows.
\end{proof}

\begin{lem} \label{l:force:term}
Let $\nu_n \in \cA_n$ be such that $\nu_n \xweakto v \nu \in \cA$ as $n \to \infty$. If
\[
  \sup_{n \geq 1} \sup_{x \geq 0} \nu_n^+ ([x, x+1]) < \infty,
\]
then (recall $g(x) = \int_x^\infty V$)
\begin{align*}
\int_{-\infty}^0 (V * \nu_n)(x) \, dx
  \xto{n \to \infty} \int_\Omega g \, d\nu.
\end{align*}
\end{lem}

\begin{proof}
We follow the proof in \cite{GarroniVanMeursPeletierScardia16} for the continuum setting and present it here in more detail. Take a continuous cut-off function 
\begin{equation} \label{psiM}  
  \psi_M : [0, \infty) \to [0,1]
  \qtext{with} \psi_M(x) = \left\{ \begin{array}{ll}
    1
    &\text{if } x \leq M  \\
    0
    &\text{if } x \geq M + 1.
  \end{array} \right.
\end{equation}
We recall from \eqref{force:rewrite} that
\begin{equation} \label{pf:12}
  \int_{-\infty}^0 (V * \nu_n)(x) \, dx
  = \int_\Omega g \, d\nu_n
  = \int_\Omega g \psi_M \, d\nu_n + \int_{M}^\infty g (1-\psi_M) \, d\nu_n,
\end{equation}
where $M > 0$ is an arbitrary constant. Since $g \psi_M \in C_c([0,\infty))$, we obtain from $\nu_n \xweakto v \nu$ that
\[
  \int_\Omega g \psi_M \, d\nu_n
  \xto{n \to \infty} \int_\Omega g \psi_M \, d\nu.
\]
For the second term in \eqref{pf:12}, we use \eqref{gnu:UB:l} to estimate
\[
  \bigg| \int_{M}^\infty g (1-\psi_M) \, d\nu_n \bigg| 
  \leq \int_{M}^\infty g \, d|\nu_n|
  \leq \e_M \sup_{x \geq 0} | \nu_n | ([x, x+1])
  \leq C \e_M,
\]
where $\e_M \to 0$ as $M \to \infty$.
By a similar argument, it follows from $\nu \in \cA$ that
\[
  \bigg| \int_{M}^\infty g (1-\psi_M) \, d\nu \bigg| 
  \leq C \e_M.
\]
Tracing these observations back to \eqref{pf:12}, we conclude 
\[
  \int_{-\infty}^0 (V * \nu_n)(x) \, dx
  \xto{n \to \infty} \int_\Omega g \psi_M \, d\nu + \int_{M}^\infty g (1-\psi_M) \, d\nu + C \e_M
  = \int_\Omega g \, d\nu + C \e_M.
\]
Since $M$ is arbitrary, Lemma \ref{l:force:term} follows.
\end{proof}

\section{Proof of Theorem \ref{t}}
\label{s:t:pf}

This section is devoted to the proof of Theorem \ref{t}. Theorem \ref{t} consists of three statements: compactness, the liminf inequality and the limsup inequality. We prove these three statements respectively in Sections \ref{s:t:pf:Cpss}, \ref{s:t:pf:limf} and \ref{s:t:pf:limp} for the power-law case $a > 0$ (see Assumption \ref{a:V}\ref{a:V:sing}). In Section \ref{s:t:pf:a0} we show that with minor modifications the proof for the logarithmic case $a = 0$ follows.

\subsection{Compactness}
\label{s:t:pf:Cpss}

Let $\nu_n \in \cA_n$ be such that $\sup_{n \geq 1} F_n(\nu_n) < \infty$. We start from the expression for $F_n(\nu_n)$ in \eqref{Fn:full:form}. Since the fourth and fifth term in \eqref{Fn:full:form} are nonnegative (recall \eqref{EL}), we may neglect them. By Lemma \ref{l:Vgam:err:term} the third term is uniformly bounded.
Hence, it is sufficient to focus on the first two terms in \eqref{Fn:full:form}, which we label $\mathsf F_n(\nu_n)$.

Recalling the regularization $V^\beta$ defined in \eqref{Vbeta}, we expand
\begin{align} \notag
  \mathsf F_n(\nu_n)
  &= \frac{1}{2} \iint_{\Delta^c} V ( x - y ) \, d \nu_n(y) d \nu_n(x)
    - \rhomin(0) \int_{-\infty}^0 (V * \nu_n) \\\notag
  &= \frac{1}{2} \iint_{\Delta^c} (V - V^\beta) ( x - y ) \, d \nu_n(y) d \nu_n(x)
    + \frac{1}{2} \iint_{\R^2} V^\beta ( x - y ) \, d \nu_n(y) d \nu_n(x) \\\label{pf:4}
  & \qquad  - \frac{n+1}{2n^2} \gamma_n^2 V^\beta(0)
    - \rhomin(0) \int_{-\infty}^0 (V * \nu_n).
\end{align}
For the second term, note from Lemma \ref{l:Tbeta}\ref{l:Tbeta:Tnu2:Vnunu} that 
\begin{align*}
  \iint_{\R^2} V^\beta ( x - y ) \, d \nu_n(y) d \nu_n(x)
  = \| T^\beta \nu_n \|_2^2.
\end{align*}
The first and third terms in \eqref{pf:4} can be bounded from below by small constants. Indeed, using $V^\beta(0) \leq C/\beta^a$, we bound the third term by
\[
  -\frac{n+1}{2n^2} \gamma_n^2 V^\beta(0)
  \geq  - C \frac{\gamma_n^2}n \beta^{-a}.
\]
For the first term in \eqref{pf:4}, we recall $W^\beta = V - V^\beta$ and expand $\nu_n = \tilde \mu_n - \tilde \rho_*$:
\begin{multline*}
  \frac{1}{2} \iint_{\Delta^c} W^\beta ( x - y ) \, d \nu_n(y) d \nu_n(x) \\
  = \frac{1}{2} \iint_{\Delta^c} W^\beta ( x - y ) \, d \tilde \mu_n(y) d \tilde \mu_n(x)
    - \int_\Omega (W^\beta * \tilde \rho_*) \, d \tilde \mu_n
    + \frac12 \int_\Omega (W^\beta * \tilde \rho_*) \tilde \rho_*. 
\end{multline*}
Since $W^\beta \geq 0$, the first and third term are nonnegative. Using $\int_\R W^\beta \leq C \beta^{1-a}$, we estimate the second term by
\begin{equation*}
  - \int_\Omega (W^\beta * \tilde \rho_*) \, d \tilde \mu_n
  \geq - \int_\Omega \bigg( \int_\R W^\beta \bigg) \| \tilde \rho_* \|_\infty \, d \tilde \mu_n
  \geq - C \gamma_n \beta^{1-a}.
\end{equation*}
Hence, the first and third term in \eqref{pf:4} are bounded from below by
\[
  - C \beta^{-a} \gamma_n \Big( \beta +  \frac{\gamma_n}n \Big).
\]
We choose $\beta$ such that this quantity is maximal. Up to a constant, this yields
\[
  \beta = \beta_n := \frac{\gamma_n}n.
\]
Then, by the assumption on $\gamma_n$ in Theorem \ref{t}, we obtain
\begin{align*} 
  C \beta_n^{-a} \gamma_n \Big( \beta_n +  \frac{\gamma_n}n \Big)
  = \underbrace{ 2C \gamma_n \Big( \frac{\gamma_n}n \Big)^{1-a} }_{=: \e_n}
  \xto{n \to \infty} 0.
\end{align*} 
Collecting these estimates, we obtain from \eqref{pf:4} that
\begin{align} \label{pf:5}
  \mathsf F_n(\nu_n)
  \geq \frac12 \| T^{\beta_n} \nu_n \|_2^2 - \e_n  - \rhomin(0) \int_{-\infty}^0 (V * \nu_n).
\end{align} 

For the fourth term in \eqref{pf:5}, we use a rougher estimate. By Lemma \ref{l:nun:bds},
\begin{equation*}
  \rhomin(0) \int_{-\infty}^0 (V * \nu_n)
  \leq C (\| T^{\beta_n} \nu_n \|_2 + 1)
\end{equation*}
for all $n$ large enough. Then, together with the first two terms in \eqref{pf:5}, we obtain
\begin{align} \label{pf:6}
  \mathsf F_n(\nu_n)
  \geq \frac12 \| T^{\beta_n} \nu_n \|_2^2 - C'.
\end{align}

By \eqref{pf:6}, $\| T^{\beta_n} \nu_n \|_2$ is bounded. This implies two useful properties. First,
\begin{align*} 
  T^{\beta_n} \nu_n \weakto \phi
\end{align*}
in $L^2(\R)$ as $n \to \infty$ along a subsequence (not relabelled) for some $\phi \in L^2(\R)$. Then, by Lemma \ref{l:nun:bds}, $\sup_{n \geq 1} |\nu_n|([0,M])$ is bounded for any $M > 0$. Hence,
\begin{align*}
  \nu_n \xweakto v \nu \in \cM(\Omega)
\end{align*}
along a further subsequence as $n \to \infty$. 

It is left to show that $\nu \in \cA$, i.e.\
\begin{align} \label{pf:14}
  \nu^- \leq \rhomin (0) \cL
  \qand
	     \sup_{x \geq 0 } \nu^+([x,x+1]) < \infty.
\end{align}
Since $\tilde \rho_*(x) = \rhomin (x/\gamma_n)$ and $\rho_*$ is continuous, it follows that
\[
  \nu_n^- = \tilde \rho_* \xweakto v \rhomin(0) \cL
\]
as $n \to \infty$. Then, together with $\nu_n \xweakto v \nu$, we obtain
\[
  \nu_n^+ 
  = \nu_n + \nu_n^- 
  \xweakto v \nu + \rhomin(0) \cL
\]
as $n \to \infty$. In particular, $\nu + \rhomin(0) \cL \geq 0$, which shows the first statement in \eqref{pf:14}. To prove the second statement, we take $x \geq 0$ arbitrary, and take a test function $\varphi \in C_c(\Omega)$ which satisfies $0 \leq \varphi \leq 1$ and 
\[
  \varphi(y) 
  = \left\{ \begin{array}{ll}
    1
    &\text{if } x \leq y \leq x+1  \\
    0
    &\text{if } y \leq x - 1 \text{ or } y \geq x + 2.
  \end{array} \right.
\] 
Then,
\begin{align*}
  \nu^+([x,x+1])
  &\leq \int_{x-1}^{x+2} \varphi \, d \nu^+
  = \int_{x-1}^{x+2} \varphi \, d \nu + \int_{x-1}^{x+2} \varphi \, d \nu^- \\
  &\leq \lim_{n \to \infty} \int_{x-1}^{x+2} \varphi \, d \nu_n + \rhomin(0) \int_{x-1}^{x+2} \varphi(x) \, dx \\
  &\leq \limsup_{n \to \infty} \bigg( \sum_{k = 0}^2 |\nu_n| ([ x + k-1, x + k ]) \bigg) + 3 \rhomin(0),
\end{align*}
which by Lemma \ref{l:nun:bds} and \eqref{pf:6} is bounded uniformly in $x$. This implies the second statement in \eqref{pf:14}.

\subsection{Liminf inequality}
\label{s:t:pf:limf}

To prove the liminf inequality in Theorem \ref{t}, let $\nu_n \xweakto v \nu$ be given.
We may assume that $F_n(\nu_n)$ is bounded along a subsequence in $n$ (not relabelled), as otherwise the liminf inequality is trivial. Then, as in the compactness proof, we obtain \eqref{pf:6}, which shows that $(T^{\beta_n} \nu_n )_n$ is bounded in $L^2(\R)$. Then, Lemma \ref{l:nun:bds} implies
\begin{equation} \label{pf:28}
  \sup_{n \geq 1} \sup_{x \geq 0} \nu_n^+([x,x+1]) \leq C.
\end{equation}

Let $\mathsf F_n$ be as in \eqref{pf:4}. First, we observe that
\begin{equation} \label{pf:29}
  \liminf_{n \to \infty} F_n(\nu_n)
  \geq \liminf_{n \to \infty} \mathsf F_n(\nu_n).
\end{equation}
Indeed, in the compactness proof we already showed that the fourth and fifth term in \eqref{Fn:full:form} are nonnegative. By \eqref{pf:28} and Lemma \ref{l:Vgam:err:term}, the third term vanishes as $n \to \infty$.

Next we bound the right-hand side in \eqref{pf:29} from below. As in the proof for the compactness, we obtain \eqref{pf:5}. Together with Lemma \ref{l:force:term} this yields
\begin{align*} 
  \liminf_{n \to \infty} \mathsf F_n(\nu_n)
  \geq \frac12 \liminf_{n \to \infty} \| T^{\beta_n} \nu_n \|_2^2 - \rhomin(0) \int_\Omega g \, d\nu.
\end{align*}
In particular, $T^{\beta_n} \nu_n \weakto \phi$ in $L^2(\R)$ as $n \to \infty$ along a subsequence to some $\phi \in L^2(\R)$. It is left to prove that $\phi = T\nu$. With this aim, let $\varphi \in C_c^\infty(\R)$ be a test function. Using \eqref{cAn:in:X10p} and Lemma \ref{l:Tbeta}\ref{l:Tbeta:BLSO} we obtain   
\begin{align} \label{pf:7}
  \int_\R \varphi (T^{\beta_n} \nu_n)
  = \int_\Omega (T^{\beta_n} \varphi) \, d\nu_n
  = \int_0^{M+1} (T^{\beta_n} \varphi) \psi_M \, d\nu_n + \int_M^\infty (T^{\beta_n} \varphi) (1-\psi_M) \, d\nu_n,
\end{align}
where $\psi_M$ is the cut-off function introduced in \eqref{psiM}, and $M > 0$ is an arbitrary constant. We pass to the limit $n \to \infty$ in both terms separately. 

For the first term in \eqref{pf:7}, we note from Lemma \ref{l:Tbeta}\ref{l:Tbeta:TTbetag:bd} that $(T^{\beta_n} \varphi) \psi_M \to (T \varphi) \psi_M$ in $C([0,M+1])$ as $n \to \infty$. Together with $\nu_n \xweakto v \nu$ this yields
\[
  \int_0^{M+1} (T^{\beta_n} \varphi) \psi_M \, d\nu_n
  \xto{n \to \infty} \int_0^{M+1} (T \varphi) \psi_M \, d\nu.
\]
For the second term in \eqref{pf:7}, we set $\opsi_M := 1-\psi_M$ and obtain from \eqref{fnu:UB:l}, Lemma \ref{l:nun:bds}, $\| T^{\beta_n} \nu_n \|_2 \leq C$ and the triangle inequality that 
\begin{align} \notag  
  \bigg| \int_M^\infty (T^{\beta_n} \varphi) \opsi_M \, d\nu_n \bigg|
  &\leq C \big\| \opsi_M T^{\beta_n} \varphi \big\|_{X_{1,2}} \sup_{x \geq 0} | \nu_n | ([x, x+1]) \\\label{pf:8}
  &\leq C' \big( \|  \opsi_M (T - T^{\beta_n}) \varphi \|_{X_{1,2}} + \| \opsi_M T \varphi \|_{X_{1,2}} \big).
\end{align}
For the second term, we set $f := T\varphi$ and compute
\begin{align*}
  \| f \opsi_M \|_{X_{1,2}}^2
  &= \int_M^\infty (x^4 + 1) f(x)^2 \opsi_M(x)^2 \, dx
    + \int_M^\infty \big( f \opsi_M \big)'(x)^2 \, dx \\
  &\leq \int_M^\infty (x^4 + 1) f(x)^2 \, dx
    + C \| \opsi_M \|_{W^{1,\infty}(\R)}^2 \int_M^\infty \big( f(x)^2 + f'(x)^2 \big) \, dx.
\end{align*} 
Since $f \in X_{1,2}$, this value vanishes as $M \to \infty$. For the first term in \eqref{pf:8}, we obtain from Lemma \ref{l:Tbeta}\ref{l:Tbeta:TTbetag:bd} that
\begin{align*}
  \| \opsi_M (T - T^{\beta_n}) \varphi  \|_{X_{1,2}}
  \leq C \| \opsi_M \|_{W^{1,\infty}(\R)} \| (T - T^{\beta_n}) \varphi \|_{X_{1,2}}
  \leq C' \beta_n^{1-a},
\end{align*}
which vanishes as $n \to \infty$ uniformly in $M$. In conclusion, by tracing these observations back to \eqref{pf:7}, we obtain
\begin{align} \label{pf:9}
  \int_\R \varphi (T^{\beta_n} \nu_n)
  \xto{n \to \infty} \int_0^{M+1} (T \varphi) \psi_M \, d\nu + c_M,
\end{align}
where $c_M \to 0$ as $M \to \infty$.

Next we pass to the limit $M \to \infty$. Since also $\sup_{x \geq 0} |\nu|([x, x+1])$ is bounded, the computation in \eqref{pf:8} shows that
\[
  \bigg| \int_M^\infty (T \varphi) \opsi_M \, d\nu \bigg|
  \leq C \| (T \varphi) \opsi_M \|_{X_{1,2}}
  \xto{M \to \infty} 0.
\]
Hence, the right-hand side in \eqref{pf:9} can be cast into
\[
  \int_0^{M+1} (T \varphi) \psi_M \, d\nu + \int_M^\infty (T \varphi) (1-\psi_M) \, d\nu + o(1)
  = \int_\R (T \varphi) \, d\nu + o(1)
\]
as $M \to \infty$. Finally, by Lemma \ref{l:T} we conclude that
$ 
  \phi = T\nu.
$ 

\subsection{Limsup inequality}
\label{s:t:pf:limp}

We first assume $\rho_*(0) > 0$ and treat the special case $\rho_*(0) = 0$ afterwards. Since $F$ is the same as in \cite{GarroniVanMeursPeletierScardia16} and only depends on $U$ through the constant $\rho_*(0)$, we may use the density result in \cite{GarroniVanMeursPeletierScardia16}. This result states that it is sufficient to prove the limsup inequality only for those $\nu \in \cA$ for which $\nu \in L^2(\R)$, $\supp \nu \subset [0,M]$ for some $M > 0$ and $\nu^- \leq \rho_*(0) - \delta$ on $[0, M]$ for some $\delta > 0$. In particular, we treat $\nu$ as a density. Since $F$ is continuous as a functional on $L^2 (\R)$, we may further assume that $\nu \in C^1(\Omega)$.

We first treat the case $\int_\Omega \nu \geq 0$, and comment on the case $\int_\Omega \nu < 0$ afterwards. To choose $\nu_n \in \cA_n$, we note from \eqref{An} that $\cA_n$ can be parametrized by $\bx \in \Omega_n$ through
\begin{equation} \label{pf:3}
  \nu_n = \frac{\gamma_n}n \sum_{i=0}^n \delta_{x_i} - \tilde \rho_*,
\end{equation}
where we recall that $\tilde \rho_*(x) = \rho_*(x / \gamma_n)$ depends on $n$.
We choose $x_i$ such that
\begin{equation} \label{pf:15}
  \int_{x_{i-1}}^{x_i} \nu(x) + \tilde \rho_*(x) \, dx = \frac{\gamma_n}n
  \qquad \text{for all } i = 1,\ldots,n.
\end{equation}
Note from $\int_\Omega (\nu + \tilde \rho_*) = \gamma_n + \int_\Omega \nu \geq \gamma_n$ that such an $\bx \in \Omega_n$ exists. By taking $n$ large enough, we may further assume that $\nu + \tilde \rho_* \geq 0$ and that
\begin{align} \label{pf:24}
  x_i \in \supp \tilde \rho_*
  \qquad \text{for all } i = 0,\ldots,n.
\end{align}

Next we prove several properties of $\nu_n$ and $\bx$. We observe from \eqref{pf:15} that
\begin{align} \label{pf:19}
  x_i - x_{i-1} 
  \geq  \frac{\gamma_n}{n \big( \| \nu \|_\infty + \| \rho_* \|_\infty \big)}
  \qquad \text{for all } i = 1,\ldots,n.
\end{align}
Hence,
\begin{equation} \label{pf:17}
  \sup_{x \geq 0} \nu_n^+ ([x, x+1])
  = \sup_{x \geq 0} \big( \# \{ x_i \in [x,x+1] \} \big) \frac{ \gamma_n }n
  \leq C.
\end{equation}
To prove
\begin{equation} \label{pf:16}
  \nu_n \xweakto v \nu \qquad \text{as } n \to \infty,
\end{equation}
we take a test function $\varphi \in C_c(\Omega)$ and set $N := \max \supp \varphi$. Taking $n$ large enough such that $\tilde \rho_* \geq \frac \delta2$ on $[0, N+1]$ and $\nu + \tilde \rho_* \geq \frac \delta2$ on $[0,M]$, we note from \eqref{pf:15} that
\begin{align} \label{pf:41}
  x_i - x_{i-1} 
  \leq \frac2\delta \frac{\gamma_n}n
  \qquad \text{for all $i$ such that } x_{i-1} < N,
\end{align}
which vanishes as $n \to \infty$.
Then, from
\begin{align*}
  \int_\Omega \varphi \, d(\nu_n - \nu)
  &= \int_0^N \varphi \, d \tilde \mu_n - \int_0^N \varphi (\nu + \tilde \rho_*) \\
  &= \sum_{i \, : \, x_{i-1} < N} \int_{x_{i-1}}^{x_i} \big( \varphi(x_{i-1}) - \varphi(x) \big) (\nu + \tilde \rho_*)(x) \, dx \\
  &\leq  \max_{i \, : \, x_{i-1} < N} \Big( \max_{x \in [x_{i-1}, x_i]} \big| \varphi(x_{i-1}) - \varphi(x) \big| \Big)
      \int_0^N (\nu + \tilde \rho_*)
\end{align*}
we obtain by the continuity of $\varphi$ that the right-hand side vanishes as $n \to \infty$. This proves \eqref{pf:16}. 

Next we prove the limsup inequality in Theorem \ref{t} for $\nu_n$ as constructed above. With this aim, we treat all five terms of $F_n$ in \eqref{Fn:full:form} separately. The latter four terms all converge as $n \to \infty$. Indeed, the fifth term in \eqref{Fn:full:form} vanishes as $n \to \infty$. By \eqref{pf:24} the fourth term equals $0$. Due to \eqref{pf:17}, Lemma \ref{l:Vgam:err:term} implies that the third term vanishes as $n \to \infty$. Due to \eqref{pf:17} and \eqref{pf:16}, Lemma \ref{l:force:term} guarantees the convergence of the second term. 

Therefore, it is sufficient to prove the limsup inequality only for the first term in \eqref{Fn:full:form}. We first show that for all $\beta > 0$ small enough
\begin{equation} \label{pf:20}  
  \limsup_{n \to \infty} \iint_{\Delta^c} V ( x - y ) \, d \nu_n(y) d \nu_n(x)
  \leq \int_{\Omega} (V^\beta * \nu) d \nu + C \beta^{1-a}.
\end{equation}
Let $I$ be the smallest integer for which $x_I \geq M$. Using that $V$ is even, we expand
\begin{subequations} \label{pf:18}
\begin{align} \notag
  &\iint_{\Delta^c} V ( x - y ) \, d \nu_n(y) d \nu_n(x) \\\label{pf:18a}
  &= \iint_{[0,x_I]^2} V^\beta ( x - y ) \, d \nu_n(y) d \nu_n(x)
    - (I+1) \Big( \frac{\gamma_n}n \Big)^2 V^\beta(0) \\\label{pf:18b}
  &\quad + \iint_{[0,x_I]^2 \setminus \Delta} W^\beta ( x - y ) \, d \nu_n(y) d \nu_n(x) \\\label{pf:18c}
  &\quad  + 2 \int_{(x_I, x_n]} \int_{[0,x_I]} V ( x - y ) \, d \nu_n(y) d \nu_n(x) \\\label{pf:18d}
  &\quad  + \iint_{(x_I, x_n]^2 \setminus \Delta} V ( x - y ) \, d \nu_n(y) d \nu_n(x) \\\label{pf:18e}
  &\quad  + 2 \int_{[0,x_n]} \int_{(x_n, \infty)} V ( x - y ) \, d \nu_n(y) d \nu_n(x) \\\label{pf:18f}
  &\quad  + \iint_{(x_n, \infty)^2 \setminus \Delta} V ( x - y ) \, d \nu_n(y) d \nu_n(x).
\end{align}
\end{subequations}
The second term in \eqref{pf:18a} is negative; we simply bound it from above by $0$. By \eqref{pf:16} and $\nu, \rho_* \in L^\infty(\R)$, we obtain that
$\nu_n |_{[0,x_I]} \xweakto v \nu |_{[0,M]}$ as $n \to \infty$. By \cite[Thm.\ 1.59]{AmbrosioFuscoPallara00} we then also have that $(\nu_n \otimes \nu_n) |_{[0,x_I]^2} \xweakto v (\nu \otimes \nu) |_{[0,M]^2}$ as $n \to \infty$. Since $V^\beta$ is continuous, this implies for the first term in \eqref{pf:18a} that
\[
  \iint_{[0,x_I]^2} V^\beta ( x - y ) \, d \nu_n(y) d \nu_n(x)
  \xto{n \to \infty} \iint_{[0,M]^2} V^\beta ( x - y ) \, d \nu(y) d \nu(x),
\]
which equals the integral in the right-hand side of \eqref{pf:20}. 

Next we bound \eqref{pf:18b}. Neglecting the negative cross terms, we estimate
\begin{multline} \label{pf:10}
  \iint_{[0,x_I]^2 \setminus \Delta} W^\beta ( x - y ) \, d \nu_n(y) d \nu_n(x) \\
  \leq \iint_{[0,x_I]^2 \setminus \Delta} W^\beta ( x - y ) \, d \nu_n^+(y) d \nu_n^+(x)
      + \iint_{[0,x_I]^2} W^\beta ( x - y ) \, d \tilde \rho_* (y) d \tilde \rho_* (x).
\end{multline}
For the second term, we recall from Lemma \ref{l:Vbeta}\ref{l:Vbeta:L1Linf} that $\int_\R W \leq C \beta^{1-a}$. Since $x_I < M+1$ for $n$ large enough, we obtain
\begin{align*}
  \iint_{[0,x_I]^2} W^\beta ( x - y ) \, d \tilde \rho_* (y) d \tilde \rho_* (x)
  \leq \int_0^{M+1} \bigg( \int_\R W^\beta \bigg) \| \rho_* \|_\infty \, d \tilde \rho_* (x)
  \leq C (M+1) \| \rho_* \|_\infty^2 \beta^{1-a}.
\end{align*}
The first term in \eqref{pf:10} is the discrete counterpart of the second term, and can be treated similarly. Relying on \eqref{pf:19} and $I = O(n / \gamma_n)$, we bound it by
\begin{multline*}
  \iint_{[0,x_I]^2 \setminus \Delta} W^\beta ( x - y ) \, d \nu_n^+(y) d \nu_n^+(x)
  = 2 \Big( \frac{\gamma_n}n \Big)^2 \sum_{i=0}^{I-1} \sum_{k=1}^{I-i} W^\beta (x_{i+k} - x_i) \\
  \leq C \frac{\gamma_n}n \sum_{i=0}^{I-1} \sum_{k=1}^\infty c \frac{\gamma_n}n W^\beta \Big( c \frac{\gamma_n}n k \Big)
  \leq C' \int_0^\infty W^\beta
  = C'' \beta^{1-a}.
\end{multline*}
Hence, \eqref{pf:18b} is bounded by $C \beta^{1-a}$ uniformly in $n$. In view of \eqref{pf:20}, it is therefore left to show that the limsup of the remaining terms in \eqref{pf:18} are nonpositive.

We start with \eqref{pf:18d}. Expanding $\nu_n = \nu_n^+ - \tilde \rho_*$ and using that $V$ is even, we rewrite 
\begin{align} \notag
  &\iint_{(x_I, x_n]^2 \setminus \Delta} V ( x - y ) \, d \nu_n(y) d \nu_n(x) \\\notag
  &= 2 \frac{\gamma_n}n \sum_{i=I+1}^n \bigg( \frac{\gamma_n}n \sum_{j=I+1}^{i-1} V(x_i - x_j) - \sum_{j=I+1}^i \int_{x_{j-1}}^{x_j} V(x_i - x) \tilde \rho_*(x) \, dx \bigg) \\\notag
  &\quad + 2 \sum_{i=I+1}^n \int_{x_{i-1}}^{x_i} \bigg( \sum_{j=I+1}^{i-1} \int_{x_{j-1}}^{x_j} V(x - y) \tilde \rho_*(y) \, dy - \frac{\gamma_n}n \sum_{j=I+1}^{i-1} V(x - x_j) \bigg) \tilde \rho_*(x) \, dx \\\notag
  &\quad + \sum_{i=I+1}^n \int_{x_{i-1}}^{x_i} \int_{x_{i-1}}^{x_i} V(x - y) \tilde \rho_*(y) \tilde \rho_*(x) \, dx \\\label{pf:23}
  &=: T_1 + T_2 + T_3.
\end{align}
Using \eqref{pf:15} and the fact that $V$ is decreasing, we obtain for the integrals inside the parentheses that
\begin{equation} \label{pf:38}
  \int_{x_{j-1}}^{x_j} V(x_i - x) \tilde \rho_*(x) \, dx
  \geq V(x_i - x_{j-1}) \int_{x_{j-1}}^{x_j} \tilde \rho_*(x) \, dx
  = \frac{\gamma_n}n V(x_i - x_{j-1})
\end{equation}
and, similarly,
\[
  \int_{x_{j-1}}^{x_j} V(x - y) \tilde \rho_*(y) \, dy
  \leq \frac{\gamma_n}n V(x - x_j)
\]
for all $x \in (x_{i-1}, x_i)$. Hence, $T_1, T_2 \leq 0$. For $T_3$, we observe from $V(x) \leq C |x|^{-a}$ and \eqref{pf:15} that
\begin{multline} \label{pf:22} 
  \int_{x_{i-1}}^{x_i} V(x - y) \tilde \rho_*(y) dy
  \leq 2 \int_0^{\gamma_n / (2 n \| \rho_* \|_\infty)} V(z) \| \rho_* \|_\infty \, dz \\
  \leq C \int_0^{\gamma_n / (2 n \| \rho_* \|_\infty)} \frac 1{z^a} \, dz
  = C' \Big( \frac{\gamma_n}n \Big)^{1-a}.
\end{multline}
Hence, $T_3 \leq C' \gamma_n^{2-a} / n^{1-a}$, which by the assumption on $\gamma_n$ in Theorem \ref{t} vanishes as $n \to \infty$. In conclusion, the limsup of \eqref{pf:18d} is nonpositive.

For \eqref{pf:18f}, we note that $(x_n, \infty)$ is disjoint with $\supp \nu_n^+$. Hence,
\begin{equation} \label{pf:21}  
  \iint_{(x_n, \infty)^2 \setminus \Delta} V ( x - y ) \, d \nu_n(y) d \nu_n(x)
  = \iint_{(x_n, \infty)^2} V ( x - y ) \tilde \rho_* (y) \tilde \rho_* (x) \, dy dx.
\end{equation}
Then, since
\begin{equation} \label{pf:30}  
  \int_{x_n}^\infty \tilde \rho_*
  = \int_\R \tilde \rho_* - \int_0^{x_n} (\tilde \rho_* + \nu) + \int_0^{x_n} \nu
  = \int_\Omega \nu 
  = C 
  \geq 0,
\end{equation}
we get $\int_{x_n / \gamma_n}^\infty \rho_* = C / \gamma_n$. Since $\rho_*$ is Lipschitz continuous, this implies that
\[
  \| \tilde \rho_* \|_{L^\infty( x_n, \infty )}
  = \| \rho_* \|_{L^\infty( x_n / \gamma_n, \infty )}
  \leq C' / \sqrt{\gamma_n}.
\]
Then, similarly to \eqref{pf:22}, we estimate the right-hand side of \eqref{pf:21} as 
\begin{align*}
  \iint_{(x_n, \infty)^2} V ( x - y ) \tilde \rho_* (y) \tilde \rho_* (x) \, dy dx
  \leq \int_{x_n}^\infty \bigg( 2 \int_0^{C \sqrt{\gamma_n}} V(z) \frac {C'}{\sqrt{\gamma_n}} \, dz \bigg)  \tilde \rho_* (x) \, dx
  \leq \frac{C''}{ \gamma_n^{3/2} } \|V\|_1,
\end{align*}
which vanishes as $n \to \infty$.

Next we treat \eqref{pf:18c}. We split $\nu_n = (\nu_n - \nu) + \nu$ in the inner integral. For the first part, we expand as in \eqref{pf:23}, 
\begin{align*} \notag
  &\int_{(x_I, x_n]} \int_{[0,x_I]} V ( x - y ) \, d (\nu_n - \nu) (y) d \nu_n(x) \\\notag
  &= \frac{\gamma_n}n \sum_{i=I+1}^n \bigg( \frac{\gamma_n}n \sum_{j=0}^I V(x_i - x_j) - \sum_{j=1}^I \int_{x_{j-1}}^{x_j} V(x_i - x) (\nu + \tilde \rho_*)(x) \, dx \bigg) \\\notag
  &\quad + \sum_{i=I+1}^n \int_{x_{i-1}}^{x_i} \bigg( \sum_{j=1}^I \int_{x_{j-1}}^{x_j} V(x - y) (\nu + \tilde \rho_*)(y) \, dy - \frac{\gamma_n}n \sum_{j=0}^I V(x - x_j) \bigg) \tilde \rho_*(x) \, dx \\
  &=: T_4 + T_5.
\end{align*}
A similar argument as that in \eqref{pf:38} yields $T_5 \leq 0$ and 
\[
  T_4 
  \leq \frac{\gamma_n}n \sum_{i=I+1}^n \frac{\gamma_n}n V(x_i - x_I).
\]
Noting from \eqref{pf:19} that 
\[
  \sum_{i=I+1}^n V(x_i - x_I)
  \leq \sum_{k=1}^\infty V( c k \gamma_n / n )
  \leq \frac1c \frac n{\gamma_n} \int_0^\infty V
  = C \frac n{\gamma_n},
\]
we obtain $\limsup_{n \to \infty} T_4 \leq 0$.

The second part of \eqref{pf:18c} equals (setting $f := V*\nu \in \Lip(\Omega)$)
\begin{equation} \label{pf:42}
  \int_{(x_I, x_n]} \int_{[0,x_I]} V ( x - y ) \nu (y) \, d \nu_n(x)
  = \int_{(x_I, x_n]} f \, d\nu_n
  = \sum_{i=I+1}^n \int_{x_{i-1}}^{x_i} \big( f(x_i) - f(x) \big) \tilde \rho_*(x) \, dx.
\end{equation}
To estimate this in absolute value, we split the sum in two parts. Let $J = \lceil n / \sqrt{\gamma_n} \rceil$ and take $n$ large enough such that $J - I - 1 \geq c n / \sqrt{\gamma_n}$ (recall that $I = O(n / \gamma_n)$) and such that \eqref{pf:41} holds for all $i \leq J$. Then,
\begin{multline*}
  \sum_{i=I+1}^J \int_{x_{i-1}}^{x_i} \big| f(x_i) - f(x) \big| \tilde \rho_*(x) \, dx
  \leq C \sum_{i=I+1}^J \int_{x_{i-1}}^{x_i} | x_i - x_{i-1} | \tilde \rho_*(x) \, dx \\
  \leq C \sum_{i=I+1}^J \Big( \frac{ \gamma_n }n \Big)^2
  \leq C' \frac{ \gamma_n^{3/2} }n
  \xto{n \to \infty} 0
\end{multline*}
and, recalling \eqref{pf:19},
\begin{align*}
  &\sum_{i=J+1}^n \int_{x_{i-1}}^{x_i} \big| f(x_i) - f(x) \big| \tilde \rho_*(x) \, dx
  \leq \sum_{i=J+1}^n \int_{x_{i-1}}^{x_i} 2 \|f\|_{L^\infty(x_{i-1}, x_i)} \tilde \rho_*(x) \, dx \\
  &\leq \sum_{i=J+1}^n 2 \frac{ \gamma_n }n V (x_{i-1} - x_I) \int_\Omega |\nu|
  \leq C \sum_{i=J+1}^n c \frac{ \gamma_n }n V \Big( c \frac{ \gamma_n }n (i - 1 - I) \Big) \\
  &\leq C \int_{c (J - I - 1) \gamma_n / n }^\infty V
  \leq C \int_{c' \sqrt{\gamma_n} }^\infty V
  \xto{n \to \infty} 0.
\end{align*}
Hence, the limsup of the second part of \eqref{pf:18c} is nonpositive. We conclude that the limsup of \eqref{pf:18c} is nonpositive.

Finally, for \eqref{pf:18e}, we observe that the inner integral
\[
  \int_{(x_n, \infty)} V ( x - y ) \, d \nu_n(y)
  = - (V * \tilde \rho_* |_{(x_n, \infty)})(x)
\]
is nonpositive and non-increasing on $[0, x_n]$. Then, using a similar argument as for \eqref{pf:18c} (simplifications are possible), we conclude that the limsup of \eqref{pf:18e} is nonpositive. This completes the proof of \eqref{pf:20}.

In conclusion, we have proved that 
\[
  \limsup_{n \to \infty} F_n(\nu_n)
  \leq \frac12 \int_{\Omega} (V^\beta * \nu) d \nu + C \beta^{1-a} - \rhomin(0) \int_\Omega g \, d\nu
\]
for all $\beta > 0$ small enough. Since $\nu \in L^2(\R)$ and  $V^\beta \to V$ in $L^1(\R)$ as $\beta \to 0$, we conclude the limsup inequality in Theorem \ref{t} by taking $\beta \to 0$ and applying \eqref{Vff:eq:TfL2}.

It is left to treat the cases $\rho_*(0) = 0$ and $\int_\Omega \nu < 0$. Since $\rho_*(0) = 0$ implies $\nu \geq 0$, these two cases are mutually exclusive. We start with the case $\rho_*(0) = 0$. We follow the proof for the case $\rho_*(0) > 0$ with minor modifications. By density, instead of assuming $\nu^- \leq \rho_*(0) - \delta$, we may assume that $\nu \geq \psi^\delta$ on $\Omega$, where $\psi^\delta$ is a monotone cut-off functions which equals $\delta$ on $[0, M-\delta]$ and $0$ on $[M,\infty)$. Under this assumption, \eqref{pf:41} does not hold for large $i$, and thus we need to construct a different argument for \eqref{pf:16} and for showing that the limsup of \eqref{pf:18c} is nonpositive. 

To prove \eqref{pf:16}, we take a test function $\varphi \in C_c(\Omega)$ and set $N := \max \supp \varphi$. From the proof of \eqref{pf:16} we observe that $\nu \geq \psi^\delta$ on $\Omega$ implies that $\int_{[0,M]} \varphi \, d(\nu_n - \nu) \to 0$ as $n \to \infty$. For the integral over $[M,\infty)$, we note from the Lipschitz continuity of $\rho_*$ and $\rhomin(0) = 0$ that
\[
  \int_M^N \tilde \rho_* 
  \leq \int_0^N \rho_* \Big( \frac x{\gamma_n} \Big) \, dx 
  \leq C \int_0^N \frac x{\gamma_n} \, dx
  \leq \frac{C'}{\gamma_n} 
  \xto{n \to \infty} 0.
\]
Then, by \eqref{pf:15},
\[
  \int_{(M,N]} d \nu_n^+ 
  \leq \frac{\gamma_n}n + \int_M^N \tilde \rho_*
  \xto{n \to \infty} 0.
\]
Hence,
\[
  \int_{(M,\infty)} \varphi \, d(\nu_n - \nu)
  = \int_{(M,N]} \varphi \, d(\nu_n^+ - \tilde \rho_*)
  \leq \| \varphi \|_\infty \int_{(M,N]} d(\nu_n^+ + \tilde \rho_*)
  \xto{n \to \infty} 0.
\]
We conclude \eqref{pf:16}.

Next we show that the limsup of \eqref{pf:18c} is nonpositive. We can follow the proof for the case $\rho_*(0) > 0$ up to the term $\int_{(x_I, x_n]} f \, d\nu_n$ in \eqref{pf:42}. Since $\rho_*(0) = 0$ implies $\nu \geq 0$, we observe that $f = V*\nu$ is nonnegative and non-increasing on $(x_I, \infty)$. Hence, by a similar argument as that in \eqref{pf:38}, it follows that $\int_{(x_I, x_n]} f \, d\nu_n \leq 0$. This concludes the proof of the limsup inequality in Theorem \ref{t} in the case $\rho_*(0) = 0$.

It is left to prove the limsup inequality for the case $\int_\Omega \nu < 0$. We largely follow the proof for the case $\int_\Omega \nu \geq 0$, and focus on the modifications. For the choice of $\nu_n$, we define $x_0, \ldots, x_J$ as in \eqref{pf:15}, where $J$ is the smallest integer at which
\[
  \int_{x_J}^\infty \tilde \rho_* \leq \frac{\gamma_n}n.
\]
For $i > J$ we set
\begin{equation} \label{pf:165}
  x_i = x_J + (i-J) \frac{\gamma_n^{3/2}}n,
\end{equation}
and take $\nu_n$ as in \eqref{pf:3}. Clearly, properties \eqref{pf:19}, \eqref{pf:17} and \eqref{pf:16} are still satisfied. To obtain the discrete equivalent of \eqref{pf:30}, we compute
\begin{align*}
  \frac{\gamma_n}n J
  = \nu_n^+([0, x_J))
  = \int_{[0, x_J)} (\nu + \tilde \rho_*)
  = \int_\Omega \nu + \int_\Omega \tilde \rho_* - \int_{x_J}^\infty \tilde \rho_*
  \geq \int_\Omega \nu + \gamma_n - \frac{\gamma_n}n.
\end{align*}
Hence, 
\begin{equation} \label{pf:39}
  n - J 
  \leq 1 + \frac n{\gamma_n} \bigg| \int_\Omega \nu \bigg|.
\end{equation}
Finally, instead of \eqref{pf:24}, we now have
\begin{equation} \label{pf:25}
  x_i \in \supp \tilde \rho_i 
  \qquad \text{for all } i = 0,\ldots,J.
\end{equation}

Similarly to the previous case, we pass to the limit $n \to \infty$ in all five terms in \eqref{Fn:full:form} separately. By Lemmas \ref{l:Vgam:err:term} and \ref{l:force:term} we obtain the same limit for the second, third and fifth term. Using \eqref{pf:25}, the fourth term equals
\begin{align*}
  \int_{(\supp \tilde \rho_*)^c} \big( U(x / \gamma_n) - C_U \big) \, d \nu_n^+(x) 
  \leq \frac{\gamma_n}n \bigg( C_U  + \sum_{i = J+1}^n \big[ U(x_i / \gamma_n) - C_U \big]^+ \bigg). 
\end{align*}
To bound the sum in the right-hand side, note from $x_J \in \supp \tilde \rho_*$ that
\[
  \sum_{i = J+1}^n \big[ U(x_i / \gamma_n) - C_U \big]^+
  \leq \sum_{i = J+1}^n \big( U(x_i / \gamma_n) - U(x_J/\gamma_n) \big).
\]
Since by \eqref{pf:39}
\[
  \frac{x_i}{\gamma_n} - \frac{x_J}{\gamma_n}
  = (i-J) \frac{\sqrt{ \gamma_n }}n
  \leq \frac1{\sqrt{ \gamma_n }}
  \xto{n \to \infty} 0,
\]
it  follows from $U \in C^1(\R)$ that the summand vanishes as $n \to \infty$. Hence, the fourth term in \eqref{Fn:full:form} also vanishes as $n \to \infty$.

We treat the first term in \eqref{Fn:full:form} similarly as in \eqref{pf:20} and \eqref{pf:18}. The modifications to \eqref{pf:18} are as follows. First, we replace $x_n$ by $x_J$ in all integration domains. Then, we replace in \eqref{pf:18e} and \eqref{pf:18f} $\nu_n$ by $-\tilde \rho_*$ in the integrals over $(x_J, \infty)$. Finally, we add two new terms (see \eqref{pf:31}) to account for $x_{J+1}, \ldots, x_n$. With these modifications, \eqref{pf:18a}--\eqref{pf:18f} can be treated analogously. The two new terms that need to be added are
\begin{align} \label{pf:31}
  2 \Big( \frac{\gamma_n}n \Big)^2 \sum_{i = J+1}^n \sum_{j = J+1}^{i-1} V(x_i - x_j)
  + 2 \frac{\gamma_n}n \sum_{i = J+1}^n \int_{[0, x_J]} V(x_i - y) \, d\nu_n(y).
\end{align}
Using \eqref{pf:165} and \eqref{pf:39} we estimate the first term by
\[
  \Big( \frac{\gamma_n}n \Big)^2 \sum_{i = J+1}^n \sum_{j = J+1}^{i-1} V(x_i - x_j)
  \leq \Big( \frac{\gamma_n}n \Big)^2 \sum_{i = J+1}^n \sum_{k = 1}^{n - J} V(k \gamma_n^{3/2} / n)
  \leq C \frac{\gamma_n}n \sum_{k = 1}^{n - J} V(k \gamma_n^{3/2} / n).
\]
For the second term in \eqref{pf:31}, we apply a similar argument as for \eqref{pf:18c}. This yields
\[
  \frac{\gamma_n}n \sum_{i = J+1}^n \int_{[0, x_J]} V(x_i - y) \, d(\nu_n - \nu)(y)
  \leq \frac{\gamma_n}n \sum_{i = J+1}^n V(x_i - x_J)
  = \frac{\gamma_n}n \sum_{k = 1}^{n-J} V(k \gamma_n^{3/2} / n)
\]
and
\[
  \frac{\gamma_n}n \bigg| \sum_{i = J+1}^n \int_{[0, x_J]} V(x_i - y) \nu(y) \, dy \bigg|
  \leq C V(x_{J+1} - x_I) \int_\Omega |\nu|
  \leq C V(c \gamma_n)
  \xto{n \to \infty} 0.
\]
Hence, \eqref{pf:31} is bounded from above by
\[
  C \frac{\gamma_n}n \sum_{k = 1}^{n-J} V(k \gamma_n^{3/2} / n)
  \leq \frac C{\sqrt{ \gamma_n }} \sum_{k = 1}^\infty \frac{ \gamma_n^{3/2} }n V \Big( k \frac{ \gamma_n^{3/2} }n \Big) 
  \leq \frac C{\sqrt{ \gamma_n }} \int_0^\infty V
  \xto{n \to \infty} 0.
\]
This completes the proof of the limsup inequality in Theorem \ref{t}.

\subsection{The case $a=0$}
\label{s:t:pf:a0}

Here we prove Theorem \ref{t} for the case $a = 0$. The proof is the same as for the case $0 < a < 1$, except for minor computational modifications. All these modifications are ramifications from the difference in the bound on $V$ in Assumption \ref{a:V}\ref{a:V:sing}, which in the current case $a = 0$ produces logarithms. The ramifications in the preliminary estimates are the statement of Lemma \ref{l:Vbeta}\ref{l:Vbeta:L1Linf}, which changes into
\begin{align} \label{a0:z}
  V^\beta(0) \leq  C |\log \beta| 
  \qand
  \int_{\R} W^\beta \leq C \beta |\log \beta|,
\end{align}
and the statement of Lemma \ref{l:Tbeta}\ref{l:Tbeta:TTbetag:bd}. By observing that $\int_0^\beta x^j |\log x| \, dx \leq C \beta^{j+1} |\log \beta|$, it follows from the proof of Lemma \ref{l:Tbeta}\ref{l:Tbeta:TTbetag:bd} that the corresponding statement becomes
\begin{align} \label{a0:y}
    \| (T - T^\beta) f \|_{X_{1,2}} 
    \leq C_f \beta |\log \beta|.
\end{align} 

In the compactness proof, by taking again $\beta_n = \gamma_n/n$, we observe from \eqref{a0:z} that the term $\e_n$ becomes
\[
  \e_n = C \frac{\gamma_n^2}n \Big| \log \frac{\gamma_n}n \Big|.
\]
Applying the asymptotic bound on $\gamma_n$ in Theorem \ref{t}, we obtain
\[
  \e_n 
  \ll C \frac1{\log n} ( \log \sqrt n + \log \sqrt{\log n} )
  \leq C,
\]
which is sufficient for continuing the argument in the proofs for the compactness and the liminf inequality.

For the liminf inequality, we only need that \eqref{a0:y} vanishes as $\beta \to 0$, which is obvious. For the limsup inequality, the bounds in \eqref{a0:z} yield that \eqref{pf:18b} is bounded by $C \beta |\log \beta|$. Hence, we may replace the term $C \beta^{1-a}$ in \eqref{pf:20} by $C \beta |\log \beta|$. Also, the bound in \eqref{pf:22} changes by \eqref{a0:z} into
\[
  \int_{x_{i-1}}^{x_i} V(x - y) \tilde \rho_*(y) dy
  \leq C \int_0^{\gamma_n / (2 n \| \rho_* \|_\infty)} |\log z| \, dz
  = C' \frac{\gamma_n}n \Big| \log \frac{\gamma_n}n \Big|.
\]
From this estimate, we obtain for the term $T_3$ in \eqref{pf:23} that 
\[
  T_3 
  \leq C \frac{\gamma_n^2}n \Big| \log \frac{\gamma_n}n \Big|
  = C' \e_n
  \xto{n \to \infty} 0,
\]
which is sufficient for the proof of the limsup inequality.

\section*{Acknowledgements}

The author gratefully acknowledges support from JSPS KAKENHI Grant Number JP20K14358.

\end{document}